\documentclass{amsart}

\usepackage{amssymb,amsmath,amsthm,amsfonts}
\usepackage{graphicx}
\usepackage{amsthm}
\usepackage{amsmath}
\usepackage{color}
\usepackage{mathabx}
\usepackage[all,cmtip]{xy}
\usepackage{tikz-cd}
\usepackage{stmaryrd}
\usepackage{hyperref}
\usepackage{algorithm}
\usepackage{algpseudocode}
\usepackage{pifont}
\usepackage[margin=1.5in]{geometry}
\usepackage{mathtools}
\usepackage[percent]{overpic}

\newtheorem{thm}{Theorem}[section]
\newtheorem{prop}[thm]{Proposition}
\newtheorem{lem}[thm]{Lemma}
\newtheorem{cor}[thm]{Corollary}
 
\newtheorem{definition}[thm]{Definition}

\newtheorem{remark}[thm]{Remark}

\newcommand{\R}{\mathbb{R}}
\newcommand{\Z}{\mathbb{Z}}

\newcommand{\C}{\mathbb{C}}

\newcommand{\op}{\mathcal{O}}
\newcommand{\cl}{\mathcal{C}}

\newcommand{\open}{\mathcal{O}}

\newcommand{\LSO}{\mathcal{L}\mathrm{Sim}_0}

\newcommand{\bbS}{\mathbb{S}}

\newcommand{\M}{\mathcal{M}}

\newcommand{\Gr}{\mathrm{Gr}_2(\mathcal{V})}
\newcommand{\St}{\mathrm{St}_2(\mathcal{V})}
\newcommand{\lSt}{\mathrm{St}_2(\mathcal{L}\C)}
\newcommand{\lGr}{\mathrm{Gr}_2(\mathcal{L}\C)}
\newcommand{\diff}{\mathrm{Diff}^+(S^1)}

\newcommand{\qi}{\textbf{i}}
\newcommand{\qj}{\textbf{j}}
\newcommand{\qk}{\textbf{k}}

\newcommand{\ls}{\mathcal{L}\bbS^1/\bbS^1}

\newcommand{\dif}{\mathrm{Diff}}
\newcommand{\Dif}{\mathrm{Diff}^+(S^1)}

\begin{document}

\title{K\"{a}hler structures on spaces of framed curves}

\author{Tom Needham}

\address{Department of Mathematics, The Ohio State University}

\email{needham.71@osu.edu}

\begin{abstract}
We consider the space $\mathcal{M}$ of Euclidean similarity classes of framed loops in $\mathbb{R}^3$. Framed loop space is shown to be an infinite-dimensional K\"{a}hler manifold by identifying it with a complex Grassmannian. We show that the space of isometrically immersed loops studied by Millson and Zombro is realized as the symplectic reduction of $\mathcal{M}$ by the action of the based loop group of the circle, giving a smooth version of a result of Hausmann and Knutson on polygon space. The identification with a Grassmannian allows us to describe the geodesics of $\M$ explicitly. Using this description, we show that $\mathcal{M}$ and its quotient by the reparameterization group are nonnegatively curved. We also show that the planar loop space studied by Younes, Michor, Shah and Mumford in the context of computer vision embeds in $\mathcal{M}$ as a totally geodesic, Lagrangian submanifold. The action of the reparameterization group on $\mathcal{M}$  is shown to be Hamiltonian and this is used to characterize the critical points of the weighted total twist functional.
\end{abstract}

\maketitle

\section{Introduction}

Let $\mathrm{Pol}(n,\vec{r})$ denote the space of $n$-edge polygons in $\R^3$ with fixed edgelengths given by $\vec{r}=(r_1,\ldots,r_n)$, where two polygons are identified if they differ by a rigid motion. In \cite{kapovich}, Kapovich and Millson realize $\mathrm{Pol}(n,\vec{r})$ as the symplectic reduction of a product of 2-spheres by the diagonal action of $\mathrm{SO}(3)$. Remarkably, essentially the same construction extends to the infinite-dimensional space of smooth curves. Millson and Zombro show in \cite{millson} that the space $\mathrm{IsoImm}(S^1,\R^3)$ of smooth arclength-parameterized loops, identified up to rigid motions, is realized as the symplectic reduction of the loop space of the 2-sphere by the rotation action of $\mathrm{SO}(3)$. Illustrating a symplectic Gelfand-Macpherson correspondence, Hausmann and Knutson show in \cite{hausmann} that $\mathrm{Pol}(n,\vec{r})$ can alternatively be realized as the symplectic reduction of the Grassmannian of 2-planes $\mathrm{Gr}_2(\C^n)$ by the natural action of $\mathrm{U}(1)^n/\mathrm{U}(1)$. This idea was taken further by Howard, Manon and Millson in \cite{howard} to identify $\mathrm{Gr}_2(\C^n)$ with a moduli space of \emph{framed} polygons called the \emph{space of spin-framed $n$-gons}. The aim of this paper is to give smooth versions of the constructions of \cite{hausmann,howard} and to show a relationship between these ideas and recent work in the field of computer vision. 

We begin by giving a construction of a K\"{a}hler structure on the space of smooth, parameterized, framed paths in $\R^3$---a \emph{framing} of a smooth curve is a choice of smooth normal unit vector field along the curve. The basic idea of the construction is to represent a framed path in $\R^3$ as a path in $\C^2$ (or the quaternions) via the Hopf map. This trick is well-known to the computer graphics community \cite{hanson} and has applications to contact geometry \cite{arnold}, but we take the novel viewpoint that this representation is a local diffeomorphism of infinite-dimensional manifolds. This locally embeds the space of framed paths into the path space of $\C^2$---a complex Fr\'{e}chet vector space denoted $\mathcal{P}\C^2$. We endow $\mathcal{P}\C^2$ with a Hermitian $L^2$ metric, and the space of framed paths is thereby shown to have a rather transparent K\"{a}hler structure. Moreover, various natural moduli spaces of framed curves are realized as symplectic reductions of $\mathcal{P}\C^2$. 

We are particularly interested in the \emph{moduli space of framed loops},
$$
\mathcal{M}:=\{\mbox{relatively framed loops in $\R^3$}\}/\mathrm{Sim},
$$
where $\mathrm{Sim}$ denotes the group of Euclidean similarities. A \emph{relative framing} of a loop is an equivalence class of framings which is determined up to a choice of initial conditions. Our first main result says that each connected component of $\mathcal{M}$ is identified in our coordinate system with an infinite-dimensional complex Grassmannian (Theorem \ref{thm:grassmannian}). This follows from the fact that the closure condition for a framed path in the complex coordinate system is simply $L^2$-orthonormality---this is in stark contrast to the traditional curvature and torsion functional coordinates on the space of Frenet-framed paths, in which the known closure characterizations are impractical to check \cite{grinevich, hwang}.

The based loop group of the circle $C^\infty(S^1,S^1)/S^1$ acts on $\M$ by \emph{frame twisting}; that is, it acts transitively on the set of relative framings of a fixed base curve. We show in Theorem \ref{thm:millson_zombro_space} that  symplectic reduction by this group action produces the Millson-Zombro space $\mathrm{IsoImm}(S^1,\R^3)$. The proof involves showing that $\mathcal{M}$ has the stucture of a principal bundle over the space of unframed loops (Theorem \ref{thm:principal_bundle}). An immediate corollary formalizes a phenomenon observed in computer graphics literature \cite{carroll, hanson2, wang}: any attempt to continuously assign a framing which works for all immersed loops will necessarily fail. 

Since $\M$ consists of \emph{parameterized} framed curves, we can also consider the action of the reparameterization group $\mathrm{Diff}^+(S^1)$. On the symplectic side, we show that the action is Hamiltonian with a momentum map that records the twisting of the vector field. This is used to characterize the critical points of a natural generalization of the classical total twist functional (Theorem \ref{thm:twist_critical_points}). The interplay of the Riemannian part of the K\"{a}hler structure of $\M$ and the action of $\diff$ shows connections with recent work on shape recognition applications. The Riemannian metric of $\M$ is $\diff$-invariant, meaning it induces a well-defined metric on the quotient space of \emph{unparameterized framed loops} $\M/\diff$. The methods of the recently developed field of \emph{elastic shape analysis} \cite{bauer, michor, srivastava, younes} can be employed to approximate geodesic distance in $\M/\diff$ (see Section \ref{subsubsec:elastic_shape_analysis}), thereby giving a shape recognition algorithm for framed loops. 

The shape recognition algorithm requires efficient computation of geodesics in $\M$. Fortunately, the fact that $\M$ is identified with a Grassmannian allows us to describe its geodesics explictly. Using this description, we show that the planar loop space studied by Younes, Michor, Shah and Mumford in \cite{younes} in the context of object recognition embeds as a totally geodesic Lagrangian submanifold of $\mathcal{M}$ (Proposition \ref{prop:tot_geod_lagrangian}). The spaces $\mathcal{M}$ and $\mathcal{M}/\mathrm{Diff}^+(S^1)$ are then shown to be nonnegatively curved (Theorem \ref{thm:sectional_curvature}), echoing prior results in the shape recognition literature for spaces of planar curves \cite{bauer,younes}.

The paper is organized as follows. Section \ref{sec:framed_curve_spaces} introduces the basic spaces of framed curves of interest and Section \ref{sec:complex_coordinates} describes the complex coordinate system for framed curve space. Sections \ref{sec:frame_twisting_action} and \ref{sec:reparameterization_action} treat the actions of $C^\infty(S^1,S^1)/S^1$ and $\mathrm{Diff}^+(S^1)$, respectively. Section \ref{sec:riemannian_geometry} is devoted to the Riemannian geometry of $\mathcal{M}$.

\subsection{Notation}\label{subsec:notation}

For a finite-dimensional manifold $M$, we will use the notation $\mathcal{P}M:=C^\infty([0,2],M)$ for the \emph{path space of $M$}---the interval $[0,2]$ is chosen as the domain of any path as a convenient normalization. Let  $\mathcal{L}M:=C^\infty(S^1,M)$ denote the \emph{loop space of $M$}. We identify $S^1$ with $[0,2]/(0\sim 2)$ so that $\mathcal{L}M$ includes naturally into $\mathcal{P}M$. The spaces $\mathcal{P}M$ and $\mathcal{L}M$ are tame Fr\'{e}chet manifolds (see \cite{hamilton} for a general reference on Fr\'{e}chet spaces) and the inclusion $\mathcal{L}M \hookrightarrow \mathcal{P}M$ embeds $\mathcal{L}M$ into $\mathcal{P}M$ as a smooth submanifold of infinite codimension. A smooth map $f:M\rightarrow N$ induces a smooth map $\mathcal{P}f:\mathcal{P}M \rightarrow \mathcal{P}N$ by the formula $(\mathcal{P}f(\gamma))(t)=f(\gamma(t))$. A similar statement holds for the loop spaces.

To distinguish from $S^1=[0,2]/(0\sim 2)$, we use $\bbS^n$ to denote the standard radius-1 $n$-sphere embedded in $\R^{n+1}$ with its induced Riemannian metric. We reserve $\left<\cdot,\cdot\right>$ and $\|\cdot\|$ for the Euclidean inner product and norm in $\R^3$, respectively. Other inner products will be given specialized notation. We will use $\R^+$ to denote the positive real numbers, considered as a Lie group under multiplication.

\section{Framed Curve Spaces}\label{sec:framed_curve_spaces}

\subsection{Framed Path Space}\label{subsec:framed_path_space}

We begin with a formal definition of framed path space.

\begin{definition}
A \emph{framed path} is a pair $(\gamma,V)$ of smooth maps $[0,2] \rightarrow \R^3$ such that the \emph{base curve} $\gamma$ is an immersion and the \emph{framing} $V$ is a unit normal vector field along $\gamma$. The \emph{moduli space of framed paths} is the quotient space
$$
\mathcal{O}:=\left\{\mbox{framed paths}\right\}/\R^3,
$$
where $\R^3$ acts on a framed path by translation.
\end{definition}

It will frequently be convenient to represent elements of $\op$ as framed paths $(\gamma,V)$ with $\gamma(0)=\vec{0}$. This representation is equivalent to choosing a global section of the $\R^3$-bundle $\{\mbox{framed paths}\} \rightarrow \op$. This convention will be assumed unless otherwise noted.

A simple but useful observation is that it is possible to represent a framed path $(\gamma,V)$ as an element of $\mathcal{P}(\mathrm{SO}(3) \times \R^+)$ via the \emph{frame map} $\mathrm{F}:\op \rightarrow \mathcal{P}(\mathrm{SO}(3) \times \R^+)$ defined by 
\begin{equation} \label{eqn:map_to_SO3}
\mathrm{F}:(\gamma,V) \mapsto \left((T,V,T\times V),\|\gamma'\|\right),  \;\; T=\frac{\gamma'}{\|\gamma'\|}.
\end{equation}
Since $\op$ contains framed paths defined up to translation, $\mathrm{F}$ is invertible. If we take elements of $\op$ to be based at the origin, then the inverse is given explicitly by
\begin{equation}\label{eqn:map_from_SO3}
\mathrm{F}^{-1}:((U,V,W),r) \mapsto \left(t\mapsto \int_0^t r\left(\tilde{t}\right)U\left(\tilde{t}\right) \; \mathrm{d}\tilde{t}, V \right).
\end{equation}
Here we are denoting an element of $\mathcal{P}\mathrm{SO}(3)$ as a triple $(U(t),V(t),W(t))$ of paths in $\R^3$ which are pairwise orthonormal for all $t$. The frame map gives an identification of $\op$  with $\mathcal{P}(\mathrm{SO}(3) \times \R^+)$ and we conclude that $\op$ naturally has the structure of a tame Fr\'{e}chet manifold. 

To simplify notation, we denote the group of Euclidean similarities by $\mathrm{Sim}:=\R^3 \times \mathrm{SO}(3) \times \R^+$ and the subgroup of similarities that fix the origin by $\mathrm{Sim}_0:=\mathrm{SO}(3) \times \R^+$.

\subsection{Framed Loop Space}\label{subsec:framed_loop_space}

The manifold $\op$ of open framed paths contains the more interesting submanifold of \emph{closed} framed loops. It will be convenient to introduce an intermediate space.

\begin{definition}
A \emph{frame-periodic framed path} is a framed path $(\gamma,V)$ such that $\gamma'$ and $V$ are closed smooth curves. If a frame-periodic framed path also satisfies $\gamma(0)=\gamma(2)$, then it is called a \emph{framed loop}. The collections of frame-periodic framed paths and framed loops, considered up to translation, will respectively be denoted $\op_{fp}$ and $\cl$.
\end{definition}

As in the case of $\op$, we will typically represent elements of $\op_{fp}$ and $\cl$ as framed paths/loops  which are based at the origin.

If $(\gamma,V)$ is a frame-periodic framed path, then its image under the frame map (\ref{eqn:map_to_SO3}) is a loop in $\mathrm{Sim}_0$. The identification $\op \approx \mathcal{P}\mathrm{Sim}_0$ restricts to an identification $\op_{fp} \approx \mathcal{L}\mathrm{Sim}_0$ and it follows that $\op_{fp} \subset \op$ is a submanifold of infinite codimension. It is apparent that the frame map restricts to give an embedding of $\mathcal{C}$ into $\LSO$ which is not surjective. In fact, the inverse frame map takes an element $((U,V,W),r)$ of $\mathcal{L}\mathrm{Sim}_0$ into $\cl$ if and only if $\int_{S^1} rU \, \mathrm{d}t = \vec{0}$---this is simply the closure condition for $\gamma$. A straightforward application of Hamilton's Implicit Function Theorem \cite[Section III, Theorem 2.3.1]{hamilton} proves the following proposition. We omit the proof, but its structure is similar to the proof of Proposition \ref{prop:arclength_parameterized_submanifolds}, which is given below.

\begin{prop}\label{prop:tilde_C_manifold}
Framed loop space $\cl$ is a codimension-3 submanifold of frame-periodic framed path space $\op_{fp}$.
\end{prop}

The identification $\op_{fp} \approx \LSO$ shows that $\op_{fp}$ has two path components. The submanifold $\cl$ therefore has at least two path components and an argument similar to the classical proof of the Whitney-Graustein theorem \cite{whitney} can be used to show that it has exactly two. The path component of a framed loop $(\gamma,V)$ with embedded base curve $\gamma$ is determined by the linking number of $\gamma$ with a small pushoff $\gamma + \epsilon V$, modulo 2. Accordingly, the path components of $\cl$ are denoted $\cl_{od}$ and $\cl_{ev}$ for odd and even self-linking number, respectively.

\section{Complex Coordinates for Framed Curve Spaces}\label{sec:complex_coordinates}

\subsection{Complex Coordinates for Framed Paths}

The goal of this section is to provide a dictionary between the framed curve spaces $\op$ and $\cl$ and various submanifolds of the path space $\mathcal{P}\C^2$. We begin by introducing some notation. Elements of the complex vector space $\mathcal{P}\C^2$ will be denoted $\Phi=(\phi,\psi)$, where $\phi,\psi \in \mathcal{P}\C$. We endow $\C^2$ with its standard Hermitian inner product $\left<\cdot,\cdot\right>_{\C^2}$ and $\mathcal{P}\C^2$ with the $L^2$ Hermitian (weak) inner product
$$
\left<\Phi_1,\Phi_2\right>_{L^2}:= \int_0^2 \left<\Phi_1(t),\Phi_2(t)\right>_{\C^2} \; \mathrm{d}t.
$$
The associated norm will be denoted $\|\cdot\|_{L^2}$. This Hermitian structure trivially gives a K\"{a}hler structure on $\mathcal{P}\C^2$ with Riemannian metric $g^{L^2}:=\mathrm{Re}\left<\cdot,\cdot\right>_{L^2}$, symplectic form $\omega^{L^2}:=-\mathrm{Im}\left<\cdot,\cdot\right>_{L^2}$ and complex structure given by pointwise multiplication by the imaginary unit $i$. Since $\omega^{L^2}$ doesn't depend on its basepoint, it is obviously closed in the sense of \cite[Section 1.4]{brylinski}. This K\"{a}hler structure restricts to the linear subspace $\mathcal{L}\C^2$ and to the open submanifolds 
$$
\mathcal{P}^\circ \C^2:=\mathcal{P}(\C^2 \setminus \{\vec{0}\}) \;\; \mbox{ and } \;\; \mathcal{L}^\circ \C^2:=\mathcal{L}(\C^2 \setminus \{\vec{0}\}).
$$
We will abuse notation and continue to use $\left<\cdot,\cdot\right>_{L^2}$, $g^{L^2}$ and $\omega^{L^2}$ to denote the restrictions of these objects to the subspaces. Abusing notation even further, we will denote the $L^2$ Hermitian inner product on $\mathcal{P}\C$ (and its restriction to any subspaces) by
$$
\left<\phi,\psi\right>_{L^2} := \int_0^2 \phi \overline{\psi} \, \mathrm{d}t,
$$
and the induced norm by $\|\cdot\|_{L^2}$.

We require the definition of a new subspace of $\mathcal{P}\C^2$. A path $\Phi \in \mathcal{P}\C^2$ is called \emph{smoothly antiperiodic} if 
$$
\left.\frac{\mathrm{d}^k}{\mathrm{d}t^k}\right|_{t=2} \Phi = -\left.\frac{\mathrm{d}^k}{\mathrm{d}t^k}\right|_{t=0} \Phi \;\;\; \mbox{for all $k=0,1,2,\ldots$}.
$$
The space of antiperiodic paths in $\C^2$, or the \emph{antiloop space of $\C^2$}, is denoted $\mathcal{A}\C^2$. The  antiloop space is a complex tame Fr\'{e}chet vector space. We define $\mathcal{A}\C$ and $\mathcal{A}(\C^2 \setminus \{\vec{0}\})=:\mathcal{A}^\circ \C^2$  similarly. There is a biholomorphism from $\mathcal{L}\C$ to $\mathcal{A}\C$ given by $\phi(t) \mapsto e^{it/2} \phi(t)$.

\begin{prop}\label{prop:kahler_structure_for_paths}
There is a smooth double covering $\mathcal{P}^\circ \C^2 \rightarrow \op$ which restricts to a smooth double covering $\mathcal{L}^\circ \C^2 \sqcup \mathcal{A}^\circ\C^2 \rightarrow \op_{fp}$. By transfer of structure, $\op$ and $\op_{fp}$ are K\"{a}hler manifolds.
\end{prop}

The proof of the proposition relies on the well-known trick of representing a framed path as a path in the quaternions via the frame-Hopf map (see, e.g., \cite{arnold, hanson}). Let $\mathbb{H}=\mathrm{span}_\R\{1,\qi,\qj,\qk\}$ denote the quaternions. The \emph{frame-Hopf map} is the map $\mathrm{Hopf}:\mathbb{H}\rightarrow \R^{3 \times 3}$ defined by
\begin{equation}\label{eqn:quaternion_hopf}
\mathrm{Hopf}(q)=(\overline{q}\qi q, \overline{q} \qj q, \overline{q} \qk q).
\end{equation}
In the above, $\overline{q}$ denotes the \emph{quaternionic conjugate} of $q$. Each entry on the right side of (\ref{eqn:quaternion_hopf}) lies in $\mathrm{span}_\R\{\qi,\qj,\qk\}$, which we identify with $\R^3$. We can identify $\C^2$ with $\mathbb{H}$ via $(z,w) \leftrightarrow q=z+w\textbf{j}$ and $i \leftrightarrow \qi$, and under this identification the frame-Hopf map is given  by the formula
\begin{equation}\label{eqn:hopf_map_complex_coords}
\mathrm{Hopf}(z,w):=\left(\begin{array}{ccc}
|z|^2-|w|^2 & 2\mathrm{Im}(zw) & -2\mathrm{Re}(zw) \\
2\mathrm{Im}(z\overline{w}) & \mathrm{Re}(z^2+w^2) & \mathrm{Im}(z^2+w^2) \\
2\mathrm{Re}(z\overline{w}) & \mathrm{Im}(-z^2+w^2) & \mathrm{Re}(z^2-w^2) \end{array}\right).
\end{equation}
It easy to check that each column of $\mathrm{Hopf}$ restricts to give a Hopf fibration $\bbS^1 \hookrightarrow \bbS^3 \rightarrow \bbS^2$.

The frame-Hopf map has several useful properties. We list a few of them in the following lemma. Each assertion follows by an elementary computation. In the lemma and throughout the rest of the paper we use $\mathrm{Hopf}_j$, $j\in\{1,2,3\}$ to denote the $j$-th column of $\mathrm{Hopf}$. 

\begin{lem}\label{lem:frame_hopf_map}
The frame-Hopf map has the following properties:
\begin{itemize}
\item[(i)] $\mathrm{Hopf}$ restricts to a map $\C^2 \setminus \{0\} \rightarrow \R^{3 \times 3}$ with $\mathrm{Hopf}(z,w)=\mathrm{Hopf}(z',w')$ if and only if $(z,w) = \pm (z',w')$. 
\item[(ii)] $\mathrm{Hopf}$ has the scaling property
$$
\mathrm{Hopf}(r\cdot(z,w))=r^2 \mathrm{Hopf}(z,w), \;\;\; r \in \R^+.
$$
In particular, each entry $\mathrm{Hopf}_j$, $j=1,2,3$, squares norms:
$$
\|\mathrm{Hopf}_j(z,w)\| = \|(z,w)\|_{\C^2}^2.
$$
\item[(iii)] The entries $\mathrm{Hopf}_j$ are mutually orthogonal and have the same norm.
\end{itemize}
\end{lem}

\begin{proof}[Proof of Proposition \ref{prop:kahler_structure_for_paths}]
It follows from Lemma \ref{lem:frame_hopf_map} that $\mathrm{Hopf}$ induces a smooth double-cover $\widehat{\mathrm{Hopf}}:\C^2 \setminus \{0\} \rightarrow \mathrm{Sim}_0$ defined by
$$
\widehat{\mathrm{Hopf}}(z,w):=\left(\frac{1}{\|(z,w)\|_{\C^2}^2}\mathrm{Hopf}(z,w),\|(z,w)\|_{\C^2}^2 \right),
$$
where $\widehat{\mathrm{Hopf}}(z,w)=\widehat{\mathrm{Hopf}}(z',w')$ if and only if $(z,w)=\pm (z',w')$.  Applying the path functor produces a smooth map $\mathcal{P}\widehat{\mathrm{Hopf}}: \mathcal{P}^\circ\C^2 \rightarrow \mathcal{P}\mathrm{Sim}_0$ satisfying $\mathcal{P}\widehat{\mathrm{Hopf}}(\Phi_1)=\mathcal{P}\widehat{\mathrm{Hopf}}(\Phi_2)$ if and only if $\Phi_1 = \pm \Phi_2$ for $\Phi_j \in \mathcal{P}^\circ\C^2$. In light of the identification $\mathcal{P}\mathrm{Sim}_0 \approx \op$ via the inverse frame map $\mathrm{F}^{-1}$, this completes the proof of the first claim. The map $\mathcal{P}\widehat{\mathrm{Hopf}}$ restricts to a double covering $\mathcal{L}^\circ \C^2 \sqcup \mathcal{A}^\circ \C^2 \rightarrow \mathcal{L}\mathrm{Sim}_0 \approx \op_{fp}$, where each factor in the disjoint union covers one of the two path components of $\mathcal{L}\mathrm{Sim}_0$.
\end{proof}

Since the maps in the proof will be used frequently throughout the rest of the paper, we will use the simplified notation
$$
\mathrm{H}:=\mathcal{P}\widehat{\mathrm{Hopf}} \;\; \mbox{ and } \;\; \widehat{\mathrm{H}}:=\mathrm{F}^{-1} \circ \mathrm{H}.
$$

\subsection{Projective Spaces}\label{sec:projective_spaces}

The fact that $\widehat{\mathrm{H}}$ is a double cover suggests that it would be useful to projectivize.

\begin{definition}
For $\mathcal{V}=\mathcal{P}\C$, $\mathcal{L}\C$ or $\mathcal{A}\C$, let $\mathrm{S}(\mathcal{V}^2)$ denote the radius-$\sqrt{2}$ $L^2$-sphere in $\mathcal{V}^2$. Let $\mathrm{Proj}_\R(\mathcal{V}^2)$ denote the \emph{projective space of real lines in $\mathcal{V}^2$}, obtained from $\mathrm{S}(\mathcal{V}^2)$ by identifying antipodal points. Let $\mathrm{Proj}^\circ_\R(\mathcal{V}^2)$ denote the open submanifold obtained as the corresponding quotient of the open submanifold $\mathrm{S}^\circ(\mathcal{V}^2):=\{\Phi \in \mathrm{S}(\mathcal{V}^2) \mid \Phi(t) \neq \vec{0} \, \forall \, t\}$.
\end{definition}

\begin{remark}\label{rmk:manifold_structures}
Adapting the usual finite-dimensional charts, one is able to show that the spheres and projective spaces defined above are Fr\'{e}chet manifolds. We will define infinite-dimensional complex projective spaces, Stiefel manifolds and Grassmannians below. These spaces are also Fr\'{e}chet manifolds, and we will refer to them as such without further comment. Moreover, one can show that the complex projective spaces and complex Grassmannians are complex Fr\'{e}chet manifolds by adapting the classical holomorphic charts.
\end{remark}

Phrasing the definition differently, $\mathrm{Proj}_\R(\mathcal{V}^2)$ is obtained as the quotient of $\mathcal{V}^2\setminus\{\vec{0}\}$ by the action of $\R \setminus \{0\}$ by pointwise multiplication. The path component $\R^+ \subset \R \setminus \{0\}$ acts on framed path space by scaling base curves: that is, $r \in \R^+$ acts on $(\gamma, V) \in \op$ according to the formula $(r \cdot (\gamma, V))(t) = (r \gamma(t), V(t))$. Part (b) of Lemma \ref{lem:frame_hopf_map} immediately implies the following.

\begin{lem}
The map $\widehat{\mathrm{H}}$ is equivariant with respect to the pointwise multiplication action of $\R \setminus \{0\}$ on $\mathcal{P}^\circ \C^2$ and the scaling action of $\R^+$ on $\op$ in the sense that $\widehat{\mathrm{H}}(r \cdot \Phi) = r^2 \cdot \widehat{\mathrm{H}}(\Phi)$.
\end{lem}

For the sake of concreteness, we will identify the quotient $\op/\R^+$ with the global cross-section consisting of framed curves whose base curve has fixed length $2$ (this normalization will be convenient later on). Then $\widehat{\mathrm{H}}$ restricts to give a double covering of $\op/\R^+$ by $\mathrm{S}^\circ(\mathcal{P}\C^2)$ with fibers of the form $\{\pm \Phi\}$. Indeed, let $\Phi \in \mathrm{S}^\circ(\mathcal{P}\C^2)$ and let $(\gamma,V)=\widehat{\mathrm{H}}(\Phi)$. Then $\|\Phi\|_{L^2}=\sqrt{2}$ implies
\begin{equation}\label{eqn:integral_length_calculation}
\mathrm{length}(\gamma)= \int_0^2 \|\gamma'(t)\| \; \mathrm{d}t = \int_0^2 \|\mathrm{Hopf}_1(\Phi(t))\| \; \mathrm{d}t = \int_0^2 \|\Phi(t)\|^2_{\C^2} \; \mathrm{d}t =2.
\end{equation}
The same argument holds for the loop space and antiloop spaces, and we conclude:

\begin{cor}\label{cor:real_projective_spaces}
The map $\widehat{\mathrm{H}}$ induces diffeomorphisms 
$$
\mathrm{Proj}_\R^\circ(\mathcal{P}\C^2) \approx \op/\R^+ \;\; \mbox{ and } \;\; \mathrm{Proj}_\R^\circ(\mathcal{L}\C^2) \sqcup \mathrm{Proj}_\R^\circ(\mathcal{A}\C^2) \approx \op_{fp}/\R^+.
$$
\end{cor}

Taking the real projectivization ignores the complex structure, and this suggests that we should further quotient by the $\bbS^1$ factor of $\C\setminus \{0\} \approx \R^+ \times \bbS^1$. More precisely, $\bbS^1 \subset \C$ acts on $\mathcal{P}\C^2$ by pointwise multiplication in each factor:
$$
(e^{i\theta} \cdot (\phi,\psi))(t):=(e^{i\theta}\phi(t),e^{i\theta}\psi(t)), \;\;\;\; e^{i\theta} \in \bbS^1 \subset \C.
$$
This action restricts to the various spheres $\mathrm{S}(\mathcal{V}^2)$.

\begin{definition}
For $\mathcal{V}=\mathcal{L}\C$ or $\mathcal{P}\C$, let $\mathrm{Proj}_\C(\mathcal{V}^2)$ denote the \emph{projective space of complex lines in $\mathcal{V}$}. The projective space is given by the quotient of $\mathrm{S}(\mathcal{V}^2)$ by the pointwise $\bbS^1$-action. Let $\mathrm{Proj}_\C^\circ(\mathcal{V}^2):=\mathrm{S}^\circ(\mathcal{V}^2)/\bbS^1$.
\end{definition}

Once again, this group action has a natural interpretation for framed curves. The action is by \emph{constant frame twisting}: $e^{i \theta} \in \bbS^1$ acts on $(\gamma,V)$ according to the formula
$$
e^{i\theta}\cdot (\gamma,V) = (\gamma, \cos(2 \theta) V + \sin(2 \theta) W), \;\;\;\; W=\frac{\gamma'}{\|\gamma'\|} \times V.
$$
The following lemma can be verified by a simple calculation.

\begin{lem}\label{lem:frame_hopf_equivariance}
The first coordinate of the frame-Hopf map $\mathrm{Hopf}_1$ is invariant under the diagonal $\bbS^1$-action on $\C^2$ by multiplication; that is, for all $e^{i\theta} \in \bbS^1$, 
$$
\mathrm{Hopf}_1(e^{i\theta}z,e^{i \theta}w) = \mathrm{Hopf}_1(z,w).
$$
The other coordinates of $\mathrm{Hopf}$ satisfy
\begin{align*}
\mathrm{Hopf}_2(e^{i\theta} (z,w)) &= \cos(2\theta) \mathrm{Hopf}_2(z,w) + \sin(2\theta) \mathrm{Hopf}_3(z,w),\\
\mathrm{Hopf}_3(e^{i\theta} (z,w)) &= -\sin(2\theta) \mathrm{Hopf}_2(z,w) + \cos(2\theta) \mathrm{Hopf}_3(z,w).
\end{align*}
It follows that the map $\widehat{\mathrm{H}}$ is equivariant with respect to the pointwise $\bbS^1$-action on $\mathcal{P}\C^2$ and the constant frame-twisting action of $\bbS^1$ on $\op$; that is, $\widehat{\mathrm{H}}(e^{i \theta}\Phi)=e^{i\theta} \cdot \widehat{\mathrm{H}}(\Phi)$
\end{lem}

An equivalence class of this $\bbS^1$-action on $\op$ will be called a \emph{relatively framed path}. A relatively framed path can be viewed as a path endowed with a framing which is well-defined up to a choice of initial conditions. A well-known example of a relative framing is the \emph{Bishop framing} \cite{bishop}, obtained for a given base curve $\gamma$ by evolving an initial vector $V(0)$ along $\gamma$ with no intrinsic twisting. Other examples of relative framings include the \emph{writhe framing} of \cite{dennis} and the \emph{constant twist minimizing framing} introduced in Section \ref{sec:principal_bundle}.

The lemma implies that the complex projective spaces correspond to relatively framed path spaces. As in the finite-dimensional case, the complex projective spaces are K\"{a}hler manifolds with Fubini-Study metrics inherited from their respective affine spaces. Summarizing, we have shown:

\begin{cor}\label{cor:projective_spaces}
The map $\widehat{\mathrm{H}}$ induces diffeomorphisms 
$$
\mathrm{Proj}_\C^\circ(\mathcal{P}\C^2) \approx \op/(\R^+ \times \bbS^1) \;\; \mbox{ and } \;\; \mathrm{Proj}_\C^\circ(\mathcal{L}\C^2) \sqcup \mathrm{Proj}^\circ_\C(\mathcal{A}\C^2) \approx \op_{fp}/(\R^+ \times \bbS^1). 
$$
It follows that the relatively framed path spaces are K\"{a}hler manifolds.
\end{cor}

\subsection{Complex Coordinates for Framed Loops}

A major benefit of the complex coordinates that we have developed for framed paths is that the necessary and sufficient conditions for the periodicity of the resulting framed path are surprisingly nice. We have the following key lemma, which can be seen as an infinite-dimensional analogue of the discussion in \cite[Section 3.5]{hausmann} regarding finite-dimensional polygon spaces.

\begin{lem}\label{lem:orthogonal}
A path $\Phi=(\phi,\psi)\in \mathcal{P}^\circ \C^2$ corresponds to a framed loop under $\widehat{\mathrm{H}}$ if and only if
\begin{itemize}
\item[(i)] The path $\Phi$ is either smoothly closed or smoothly antiperiodic and
\item[(ii)] the maps $\phi$ and $\psi$ have the same $L^2$ norm and are $L^2$-orthogonal:
$$
\int_0^2 |\phi|^2 \; \mathrm{d}t = \int_0^2 |\psi|^2 \; \mathrm{d}t  \;\; \mbox{ and } \;\; \int_0^2 \phi\overline{\psi} \; \mathrm{d}t= 0.
$$
\end{itemize}
\end{lem}

\begin{proof}
It follows from Proposition \ref{prop:kahler_structure_for_paths} that a path $\Phi$ satisfies $\widehat{\mathrm{H}}(\Phi) \in \op_{fp}$ if and only if $\Phi \in \mathcal{L}^\circ\C^2 \sqcup \mathcal{A}^\circ \C^2$. In this case, let $((T,V,W),r):=\mathrm{H}(\Phi) \in \mathcal{L}\mathrm{Sim}_0$. Then the frame periodic path $\widehat{\mathrm{H}}(\Phi)=\mathrm{F}^{-1}((T,V,W),r) \in \op_{fp}$ lies in the submanifold $\cl \subset \op_{fp}$ of framed loops if and only if  the closure condition $\int_{S^1} rT \; \mathrm{d}t = \vec{0}$ holds. We see from the formula for the frame-Hopf map given in (\ref{eqn:hopf_map_complex_coords}) that the closure condition is written in terms of $\phi$ and $\psi$ as 
$$
\int_0^2 |\phi|^2-|\psi|^2 \; \mathrm{d}t = \int_0^2 2\mathrm{Im}(\phi\overline{\psi}) \; \mathrm{d}t = \int_0^2 2\mathrm{Re}(\phi\overline{\psi}) \; \mathrm{d}t = 0.
$$
This is equivalent to the conditions given in (ii).
\end{proof}

This lemma has a symplectic interpretation. In the following, let $\mathcal{V}=\mathcal{L}\C$ or $\mathcal{A}\C$. The group $\mathrm{U}(2)$ acts by isometries on the Hermitian vector space $\mathcal{V}^2$ by pointwise right multiplication. This action is also Hamiltonian with momentum map
\begin{align*}
\mu_{\mathrm{U}(2)}: \mathcal{V}^2 &\rightarrow \mathfrak{u}(2) \left(\approx \mathfrak{u}(2)^\ast \right) \\
(\phi,\psi) &\mapsto  i\left(\begin{array}{cc}
\left<\phi,\phi\right>_{L^2} & -\left<\phi,\psi\right>_{L^2} \\
-\left<\psi,\phi\right>_{L^2} & \left<\psi,\psi\right>_{L^2} \end{array}\right).
\end{align*}
The conditions in part (ii) of Lemma \ref{lem:orthogonal} can then be phrased in terms of the entries of the entries of $\mu_{\mathrm{U}(2)}(\Phi)$.

\begin{definition}
For $\mathcal{V}=\mathcal{L}\C$ or $\mathcal{A}\C$, the \emph{Stiefel manifold of $L^2$-orthonormal 2-frames in $\mathcal{V}$} is the Fr\'{e}chet manifold
$$
\mathrm{St}_2(\mathcal{V}):=\left\{(\phi,\psi) \in \mathcal{V}^2 \mid \left<\phi,\phi\right>_{L^2}=\left<\psi,\psi\right>_{L^2}=1  \mbox{ and } \left<\phi,\psi\right>_{L^2} = 0 \right\}=\mu_{\mathrm{U}(2)}^{-1}\left(\begin{array}{cc}
i & 0 \\
0 & i \end{array}\right).
$$ 
Let $\mathrm{St}_2^\circ(\mathcal{V})$ denote the open submanifold
$$
\mathrm{St}^{\circ}_2(\mathcal{V}):=\{(\phi,\psi) \in \mathrm{St}_2(\mathcal{V}) \mid (\phi(t),\psi(t)) \neq (0,0) \; \forall \, t\}.
$$
\end{definition}

The scaling action of $\R^+$ on $\op$ restricts to the space of closed framed loops $\cl$. We similarly identify the quotient $\cl/\R^+$ with the space of closed loops of fixed length $2$.

\begin{prop}\label{prop:stiefel_manifold}
The restriction of $\widehat{\mathrm{H}}$ gives double coverings
$$
\mathrm{St}_2^\circ(\mathcal{L}\C)  \xrightarrow{\times 2}  \cl_{od}/\R^+ \;\; \mbox{ and } \;\; \mathrm{St}_2^\circ(\mathcal{A}\C) \xrightarrow{\times 2} \cl_{ev}/\R^+
$$
\end{prop}

\begin{proof}
A path $\Phi \in \mathcal{P}^\circ \C^2$ lies in one of the Stiefel manifolds if and only if it satisfies the conditions of Lemma \ref{lem:orthogonal} and $\|\Phi\|_{L^2}=\sqrt{2}$. The calculation (\ref{eqn:integral_length_calculation}) shows that this is equivalent to $\widehat{\mathrm{H}}(\Phi) \in \cl/\R^+$. To see that $\widehat{\mathrm{H}}$ maps the path components as claimed it suffices to check an example. For
$$
\Phi(t)=\frac{1}{\sqrt{2}}(\cos(\pi t)-i\sin (\pi t), 1) \in \mathrm{St}_2^\circ(\mathcal{L}\C)
$$
we have $\widehat{\mathrm{H}}(\Phi)=(\gamma,V)$, where
$$
\gamma(t)=\frac{1}{\pi}(0,\cos(\pi t)-1,\sin(\pi t)), \;\;\; V(t)=(-\sin(\pi t), \cos^2(\pi t), \cos(\pi t) \sin(\pi t)).
$$
Thus the image of $\Phi$ is a framed circle with linking number 1.
\end{proof}

\begin{definition}
For $\mathcal{V}=\mathcal{L}\C$ or $\mathcal{A}\C$, the \emph{Grassmannian} is the Fr\'{e}chet manifold $\mathrm{Gr}_2(\mathcal{V})$ of 2-dimensional complex subspaces of $\mathcal{V}$. It can be represented as the quotient space $\mathrm{Gr}_2(\mathcal{V}) := \mathrm{St}_2(\mathcal{V})/\mathrm{U}(2)$. Let $\mathrm{Gr}_2^\circ(\mathcal{V})$ denote the open submanifold $\mathrm{Gr}_2^\circ(\mathcal{V}):=\mathrm{St}_2^\circ(\mathcal{V})/\mathrm{U}(2)$. 
\end{definition}

The Grassmannian $\Gr$ is a complex manifold (see Remark \ref{rmk:manifold_structures}). Moreover, $\Gr$ is obtained as the symplectic reduction of $\mathcal{V}^2$ by the isometric action of $\mathrm{U}(2)$ and therefore inherits a natural K\"{a}hler structure. We will use double slash notation for symplectic reductions for the rest of the paper; e.g., $\Gr=\mathcal{V}^2 \sslash \mathrm{U}(2)$.

The next lemma shows that the action of the subgroup $\mathrm{SU}(2)\subset \mathrm{U}(2)$ on $\mathcal{P}\C^2$ corresponds to the rotation action of $\mathrm{SO}(3)$ on framed path space; that is, an element $A$ of the group $\mathrm{SO}(3)$ acts on a framed path or loop $(\gamma,V)$  pointwise by the formula $A \cdot (\gamma,V)(t):=(A \cdot \gamma(t),A \cdot V(t))$. For $U \in \mathrm{SU}(2)$, we define $\mathrm{Hopf}(U) \in \R^{3 \times 3}$ by identifying $\mathrm{SU}(2)$ with $\bbS^3 \subset \C^2$ via 
$$
\left(\begin{array}{cc}
z & w \\
-\overline{w} & \overline{z} \end{array} \right) \leftrightarrow (z,w).
$$ 
The first statement of the lemma is well-known (see, e.g.,  \cite[Section I.1.4]{gelfand}). 

\begin{lem}\label{lem:su2_action}
The restriction of $\mathrm{Hopf}$ to $\mathrm{SU}(2) \subset \C^2$ gives an anti-homomorphic double cover of $\mathrm{SO}(3)$. It follows that the map $\widehat{\mathrm{H}}$ has the property that for $\Phi \in \mathcal{P}^\circ\C^2$ and $U \in \mathrm{SU}(2)$, 
$$
\widehat{\mathrm{H}}(\Phi \cdot U) = \mathrm{Hopf}(U) \cdot \widehat{\mathrm{H}}(\Phi).
$$
\end{lem}

This brings us to our first main result and to the space which the rest of the paper will be primarily concerned with, the \emph{moduli space of framed loops}
$$
\mathcal{M}:=\mathcal{C}/(\mathrm{Sim}_0 \times \bbS^1) = \{\mbox{relatively framed loops in $\R^3$}\}/\mathrm{Sim}.
$$
The moduli space $\mathcal{M}$ has two path components corresponding to the components of $\mathcal{C}$. We correspondingly denote these components $\mathcal{M}_{ev}$ and $\mathcal{M}_{od}$. For most of the paper we will denote elements of $\M$ by $[\gamma,V]$, where $(\gamma,V)$ is a framed loop with $\gamma(0)=\vec{0}$ and $\mathrm{length}(\gamma)=2$. The brackets denote the equivalence class of $(\gamma,V)$ under the action of $\mathrm{SO}(3) \times \bbS^1$ by rotations and constant frame twisting.

\begin{thm}\label{thm:grassmannian}
The map $\widehat{\mathrm{H}}$ induces diffeomorphisms
$$
\mathrm{Gr}_2^\circ(\mathcal{L}\C) \approx \mathcal{M}_{od} \;\; \mbox{ and } \;\; \mathrm{Gr}_2^\circ(\mathcal{A}\C) \approx \mathcal{M}_{ev}.
$$
By transfer of structure, $\mathcal{M}$ is a K\"{a}hler manifold.
\end{thm}

\begin{proof}
Lemma \ref{lem:su2_action} implies that passing from $\St$ to the quotient $\St/\mathrm{SU}(2)$ has the effect of modding out by the rotation action of $\mathrm{SO}(3)$ on $\cl/\R^+$. As we have already seen, the $\mathrm{U}(1) \approx \bbS^1$ factor of $\mathrm{U}(2)\approx \mathrm{SU}(2) \times \mathrm{U}(1)$ corresponds to the constant frame twisting action on $\cl$. Further quotienting by $\bbS^1$, we obtain $\Gr = \St/\mathrm{U}(2)$. Modding out by $\bbS^1$ in particular identifies antipodal paths, thus the double cover $\widehat{\mathrm{H}}:\mathrm{St}^\circ_2(\mathcal{L}\C) \xrightarrow{\times 2} \cl_{od}$ induces a diffeomorphism $\mathrm{Gr}_2^\circ(\mathcal{L}\C) \approx \M_{od}$. Similarly, we have an induced diffeomorphism $\mathrm{Gr}_2^\circ(\mathcal{A}\C) \approx \mathcal{M}_{ev}$.
\end{proof}

We will denote elements of $\Gr$ by $[\Phi]$ for $\Phi=(\phi,\psi) \in \St$. The tangent spaces to $\St$ are
$$
T_\Phi \St = \left\{(\delta \phi, \delta \psi) \in \mathcal{V}^2 \mid g^{L^2}(\phi,\delta \phi) = g^{L^2}(\psi,\delta \psi) = \left<\phi, \delta \psi \right>_{L^2} - \left<\delta \phi, \psi\right>_{L^2} = 0\right\}.
$$
It will frequently be convenient to represent the tangent space to $[\Phi] \in \Gr$ as the \emph{horizontal tangent space} to $\Phi \in \St$---this is the subspace of tangent vectors which are $g^{L^2}$-orthogonal to the $\mathrm{U}(2)$-orbit of $\Phi$ given by
$$
T_\Phi^{hor} \St = \left\{(\delta \phi, \delta \psi) \in \mathcal{V}^2 \mid \left<\phi,\delta \phi\right>_{L^2} = \left<\psi, \delta \psi\right>_{L^2} = \left<\phi, \delta \psi \right>_{L^2} = \left<\psi, \delta \phi\right>_{L^2} = 0 \right\}.
$$

The identifications of this section are summarized below. The spaces in the first column have K\"{a}hler structures.

\begin{equation*}
\begin{aligned}[c]
\mathcal{P}^\circ \C^2 &\xrightarrow{\times 2} \op \\
\mathcal{L}^\circ \C^2 \sqcup \mathcal{A}^\circ \C^2 &\xrightarrow{\times 2} \op_{fp} \\
\mathrm{Proj}_\C^\circ(\mathcal{P}\C^2) &\approx \op/(\R^+ \times \bbS^1)\\
\mathrm{Proj}_\C^\circ(\mathcal{L}\C^2) \sqcup \mathrm{Proj}_\C^\circ (\mathcal{A}\C^2) &\approx \op_{fp}/(\R^+ \times \bbS^1)\\
\mathrm{Gr}_2^\circ(\mathcal{L}\C) \sqcup \mathrm{Gr}_2^\circ(\mathcal{A}\C) &\approx \mathcal{M}
\end{aligned}
\qquad \hspace{-0.3in} \qquad
\begin{aligned}[c]
\mathrm{Proj}_\R^\circ(\mathcal{P}\C^2)  &\approx \op/\R^+ \\
\mathrm{Proj}_\R^\circ(\mathcal{L}\C^2) \sqcup \mathrm{Proj}_\R^\circ (\mathcal{A}\C^2) &\approx \op_{fp}/\R^+\\
 \mathrm{St}_2^\circ(\mathcal{L}\C) \sqcup \mathrm{St}_2^\circ(\mathcal{A}\C) &\xrightarrow{\times 2} \cl/\R^+
\end{aligned}
\end{equation*}

\subsection{Elastic Shape Analysis}\label{subsubsec:elastic_shape_analysis}

We now take a brief detour to discuss the field of Elastic Shape Analysis, a novel approach to shape recognition that has been developed over the last decade \cite{bauer, michor, srivastava,  younes}. In this discussion we restrict our attention to shape recognition for loops in $\R^d$, although the methods of Elastic Shape Analysis have also been applied to curves in manifolds \cite{lebrigant}, surfaces in $\R^3$ \cite{kurtek}, et cetera. The main idea is to endow the space of immersions $\mathrm{Imm}(S^1,\R^d)$ with a Riemannian metric which is invariant under the reparameterization action of $\mathrm{Diff}^+(S^1)$. Then the metric descends to a well-defined metric on the quotient $\mathrm{Imm}(S^1,\R^d)/\mathrm{Diff}^+(S^1)$. Geodesic distance with respect to this Riemannian metric is then interpreted as a measure of dissimilarity between shapes. In practice, the geodesic distance between the shapes represented by parameterized loops $\gamma_1$ and $\gamma_2$  is typically computed by finding (or approximating) the infimum of geodesic distance between the $\mathrm{Diff}^+(S^1)$-orbits of $\gamma_1$ and $\gamma_2$ in the total space; that is, by computing
\begin{equation}\label{eqn:geod_dist_ESA}
\inf_{\rho_1,\rho_2 \in \mathrm{Diff}^+(S^1)} \mathrm{dist}(\gamma_1 \circ \rho_1,\gamma_2 \circ \rho_2) = \inf_{\rho \in \mathrm{\diff}} \mathrm{dist}(\gamma_1,\gamma_2 \circ \rho),
\end{equation}
where $\mathrm{dist}$ is geodesic distance in $\mathrm{Imm}(S^1,\R^d)$ and the equality follows from the assumption that $\diff$ acts by isometries. In this applied field, efficient computability of the distance (\ref{eqn:geod_dist_ESA}) is paramount and one therefore seeks a Riemannian metric on immersion space for which geodesics are easily computable. 

The \emph{elastic metrics} of Mio, Srivastava and Joshi \cite{mio} form a 2-parameter family of metrics $g^{a,b}$ on $\mathrm{Imm}(S^1,\R^2)$ with theoretically convenient properties. These are given explicitly for $\gamma \in \mathrm{Imm}(S^1,\R^2)$, $\delta \gamma_1, \delta \gamma_2 \in T_\gamma \mathrm{Imm}(S^1,\R^2) \approx C^\infty(S^1,\R^2)$ and $a,b > 0$ by the formula
\begin{equation}\label{eqn:elastic_metrics}
g^{a,b}_\gamma(\delta \gamma_1,\delta \gamma_2) = \int_{S^1} a \left<\frac{d}{ds} \delta \gamma_1,T \right>\left<\frac{d}{ds} \delta \gamma_2, T\right> + b \left<\frac{d}{ds} \delta \gamma_1,N \right>\left<\frac{d}{ds} \delta \gamma_2, N\right> \; \mathrm{d}s.
\end{equation}
In the above, $d/ds$ denotes derivative with respect to arclength, $\mathrm{d}s$ denotes arclength measure, $T=\gamma'/\|\gamma'\|$ is the unit tangent to $\gamma$ and $N$ is the oriented unit normal. In \cite{younes}, Younes, Michor, Shah and Mumford show that, with parameter choice $a=b$, the space $\mathrm{Imm}(S^1,\R^2)/\mathrm{Sim}$ of Euclidean similarity classes of immersions is locally isometric to an infinite-dimensional Grassmannian with a natural $L^2$ metric. This is remarkable because it allows for computation of explicit geodesics in $\mathrm{Imm}(S^1,\R^2)/\mathrm{Sim}$, whence the infimum procedure in (\ref{eqn:geod_dist_ESA}) can be approximated very efficiently via a dynamic programming algorithm. The relationship between our space $\M$ and the work of Younes et.\ al.\ will be expounded upon in Section \ref{sec:totally_geodesic_submanifolds}.

\subsection{Induced Geometric Structures}\label{sec:induced_geometric_structures}

Proposition \ref{prop:kahler_structure_for_paths} states that $\op$ inherits a K\"{a}hler structure from $\mathcal{P}^\circ \C^2$ by transfer of structure under the local diffeomorphism $:\widehat{\mathrm{H}}:\mathcal{P}^\circ \C^2 \rightarrow \op$. We will give explicit formulas for the various parts of the K\"{a}hler structure here.

For $(\gamma,V) \in \op$ and $(\delta \gamma_j, \delta V_j) \in T_{(\gamma,V)} \op$, $j=1,2$, we define a metric $g^\op$ on $\op$ by the formula 
$$
g^\op_{(\gamma,V)}\left((\delta \gamma_1, \delta V_1),(\delta \gamma_2, \delta V_2)\right) = \frac{1}{4} \int_0^2 \left<\frac{d}{ds}\delta \gamma_1,\frac{d}{ds}\delta \gamma_2\right> + \left<\delta V_1,W\right> \left<\delta V_2,W\right> \; \mathrm{d}s,
$$
where $\frac{d}{ds}=\frac{1}{\|\gamma'\|} \frac{d}{dt}$ denotes arclength derivative and $\mathrm{d}s=\|\gamma'\| \mathrm{d}t$ denotes arclength measure. The factor of $\frac{1}{4}$ is included only as a convenient normalization whose utility is made apparent by the following proposition. The proof of the proposition is a tedious but essentially straightforward calculation, so we omit it.

\begin{prop}\label{prop:pullback_metric}
The pullback of $g^\op$ to $\mathcal{P}^\circ \C^2$ is $g^{L^2}$.
\end{prop}

Note that $g^\op$  is invariant under the rotation action of $\mathrm{SO}(3)$ and the constant frame twisting action of $\bbS^1$, so  it descends to a well-defined metric on the various quotient spaces $\op/(\R^+ \times \bbS^1)$, $\M$, et cetera. More importantly from the perspective of applications, the metric is also invariant under the reparameterization action (by composition) of the orientation-preserving diffeomorphism group $\mathrm{Diff}^+([0,2])$ on $\op$ (and by the reparameterization action of $\mathrm{Diff}^+(S^1)$ on the various loop spaces). This property is essential for elastic shape analysis applications, since the goal there is to compute geodesic distance in the quotient $\op/\mathrm{Diff}^+([0,2])$ via the process outlined in Section \ref{subsubsec:elastic_shape_analysis}.  

One might notice that the metric $g^\op$ bears a strong resemblance to the elastic planar metrics (\ref{eqn:elastic_metrics}). Indeed, a natural extension of the planar elastic metrics to spaces of framed curves is given by the four parameter family of \emph{framed curve elastic metrics}
\begin{align*}
&g^{a,b,c,d}_{(\gamma,V)}((\delta \gamma_1, \delta V_1),( \delta \gamma_2, \delta V_2)) \\
&\hspace{.3in} = \frac{1}{4} \int_0^2 a\left<\frac{d}{ds} \delta \gamma_1,V\right>\left<\frac{d}{ds} \delta \gamma_2,V\right> + b\left<\frac{d}{ds} \delta \gamma_1,W\right>\left<\frac{d}{ds} \delta \gamma_2,W\right> \\
&\hspace{1.6in}  + c\left<\frac{d}{ds} \delta \gamma_1,T\right>\left<\frac{d}{ds} \delta \gamma_2,T\right>  + d\left<\delta V_1,W\right> \left<\delta V_2,W\right> \, \mathrm{d}s,
\end{align*}
where $T=\gamma'/\|\gamma'\|$ and $W=T \times V$. Where the planar curve elastic metrics have terms comparing bending and stretching deformations of a variation, the framed curve elastic metrics have terms comparing two types of bending deformation, stretching deformation and twisting deformation (repectively). The reader can check that $g^\op=g^{1,1,1,1}$.

We now move on to the induced complex structure of $\op$. Using quaternionic notation, any variation $\delta q$ of $q \in \mathcal{P}\mathbb{H}$ can be decomposed as
\begin{equation}\label{eqn:quaternionic_variation}
\delta q = \lambda_1 \frac{q}{2} + \lambda_2 \frac{\qi q}{2} + \lambda_3 \frac{\qj q}{2} + \lambda_4 \frac{\qk q}{2}
\end{equation}
for some $\lambda_j \in \mathcal{P}\R$. The natural complex structure of $\mathcal{P}\C^2 \approx \mathcal{P}\mathbb{H}$ can be obtained by extending the relations
$$
i \cdot \frac{q}{2} = \frac{\qi q}{2},\;\; i \cdot \frac{\qj q}{2} = \frac{\qk q}{2}, \;\; i^2=-\mathrm{Id}
$$
over combinations of the form (\ref{eqn:quaternionic_variation}). We will show that it is easiest to understand the induced structure on $\op$ from this perspective.

We define four \emph{basic variations} of a framed path $(\gamma,V)$
$$
X_{st}:=(\delta \gamma_{st},\delta V_{st}), \; X_{tw}:=(\delta \gamma_{tw}, \delta V_{tw}), \; X_{b_1}:=(\delta \gamma_{b_1}, \delta V_{b_1}), \mbox{ and } X_{b_2}:=(\delta \gamma_{b_2}, \delta V_{b_2})
$$
as the variations induced by the formulas
\begin{align*}
(\delta \gamma'_{st},\delta V_{st}) &= (\|\gamma'\| T,0) = \mbox{ stretch the  tangent vector, leave $V$ unchanged}\\
(\delta \gamma'_{tw},\delta V_{tw}) &= (0, -W) =\mbox{ twist the frame in the negative direction around $\gamma$} \\
(\delta \gamma'_{b_1},\delta V_{b_1}) &= (\|\gamma'\| W, 0) = \mbox{ bend the frame in the negative direction around $V$} \\
(\delta \gamma'_{b_2},\delta V_{b_2}) &= (-\|\gamma '\|  V, T) = \mbox{ bend the frame in the negative direction around $W$}.
\end{align*}
In each case, $\delta \gamma_\bullet$ is obtained from $\delta \gamma'_\bullet$ by taking the antiderivative which is based at $\vec{0}$. The basic variations are tangent to $\op$ at $(\gamma,V)$. Indeed, the tangent space $T_{(\gamma,V)} \op$ consists of variations $(\delta \gamma, \delta V)$ satisfying the following three pointwise constraints:
\begin{equation}\label{eqn:tangent_space}
\delta\gamma(0) = \vec{0}, \;\; \left<\frac{d}{dt}\delta\gamma,V\right>+\left<\frac{d}{dt}\gamma,\delta V \right> =0 \; \mbox{ and } \; \left<\delta V,V\right>=0.
\end{equation}
The constraints correspond to basepoint preservation, orthogonality preservation and normality preservation, respectively. Each basic variation satisfies these constraints. Any variation of $(\gamma,V)$ can be written in the form
\begin{equation}\label{eqn:functional_linear_combination}
\lambda_{st} X_{st} + \lambda_{tw} X_{tw} + \lambda_{b_1} X_{b_1}+ \lambda_{b_2} X_{b_2}
\end{equation}
for some real-valued maps $\lambda_\bullet$, so we define an almost complex structure $J^\op$ on $\op$ by extending the relations
$$
J^\op \cdot X_{st} = X_{tw}, \;\;\; J^\op \cdot X_{b_1} = X_{b_2} \;\; \mbox{ and } \;\; (J^\op)^2=-\mathrm{Id}
$$
over combinations of the form (\ref{eqn:functional_linear_combination}). The basic variations and almost complex structure are depicted in Figure \ref{fig:almost_complex_structure}.

A simple calculation shows that $D\widehat{\mathrm{H}}(q)$ gives a correspondence
$$
q/2 \mapsto X_{st}, \; \qi q/2  \mapsto X_{tw}, \; \qj q/2 \mapsto X_{b_1}, \; \qk q/2 \mapsto X_{b_2}.
$$
This proves the following proposition.

\begin{figure}
\centering
\begin{overpic}[width=1\textwidth,tics=5]{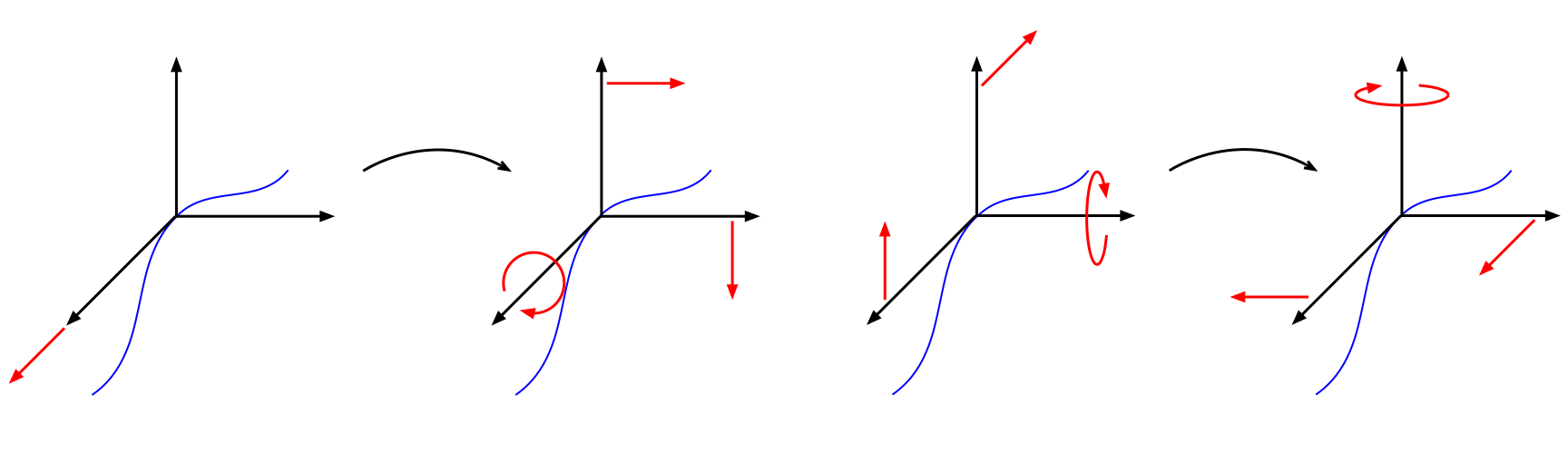}
 \put (5,13) {\large$T$}
  \put (19,18) {\large$V$}
    \put (7,25) {\large$W$}
 \put (10,10) {\textcolor{blue}{\huge$\displaystyle\gamma$}}
  \put (1,3) {\textcolor{red}{\large$\displaystyle\delta\gamma'_{st}$}}
  \put (26,22) {\large$J^\op$}
    \put (27,8) {\textcolor{red}{\large$\displaystyle\delta\gamma'_{tw}$}}
      \put (45,8) {\textcolor{red}{\large$\displaystyle\delta V_{tw}$}}
 \put (54,18) {\textcolor{red}{\large$\displaystyle\delta\gamma'_{b_1}$}}
  \put (68,10) {\textcolor{red}{\large$\displaystyle\delta V_{b_1}$}}
    \put (78,22) {\large$J^\op$}
   \put (78,7) {\textcolor{red}{\large$\displaystyle\delta\gamma'_{b_2}$}}
  \put (94,9) {\textcolor{red}{\large$\displaystyle\delta V_{b_2}$}}
\end{overpic}
\caption{The pointwise complex structure of $\op$. \label{fig:almost_complex_structure}}
\end{figure}
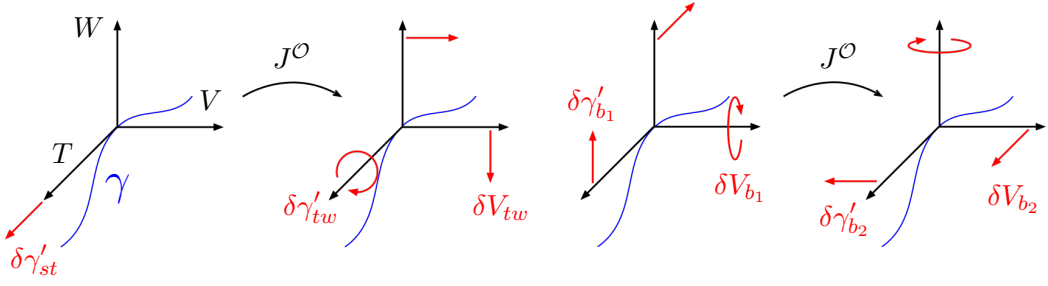

\begin{prop}\label{prop:complex_structure}
The map $\widehat{\mathrm{H}}$ intertwines the complex structure $i$ on $\mathcal{P}^\circ \C$ with the almost complex structure $J^\op$ on $\op$. It follows that $J^\op$ is integrable.
\end{prop}

Finally, we define a non-degenerate, closed 2-form $\omega^\op$ on $\op$ by the formula $\omega^\op_\cdot(\cdot,\cdot)=g^\op_\cdot(J^\op \cdot, \cdot)$. The conclusion of this subsection is that $(g^\op,\omega^\op,J^\op)$ is the K\"{a}hler structure on $\op$ corresponding to the natural K\"{a}hler structure of $\mathcal{P}^\circ \C^2$. It follows from Corollary \ref{cor:projective_spaces} and Theorem \ref{thm:grassmannian} that the spaces $\op/(\R^+ \times \bbS^1)$, $\op_{fp}/(\R^+ \times \bbS^1)$ and $\M$ inherit K\"{a}hler structures as well, as each of these spaces can be viewed as a K\"{a}hler reduction of $\op$. We will abuse notation and continue to use $g^\op$, $J^\op$ and $\omega^\op$ to denote the induced structures on these spaces. 

\section{The Frame Twisting Action}\label{sec:frame_twisting_action}

\subsection{Isometric Immersion Space}\label{subsec:smooth_loop_spaces}

In this section we introduce an action of the based loop space $\ls$ on $\M$, with the main goal being to use this action to relate $\M$ to the isometric immersion space of Millson and Zombro \cite{millson}. More precisely, let
$$
\mathrm{IsoImm}(S^1,\R^3):=\left\{\gamma \in \mathrm{Imm}(S^1,\R^3) \mid \|\gamma'(t)\|=1 \mbox{ for all $t \in S^1$}, \; \gamma(0)=\vec{0}\right\}/\mathrm{SO}(3)
$$
denote the \emph{moduli space of isometric immersions}. Recall from the introduction that it is shown in \cite{millson} that this space is realized as the symplectic reduction of $\mathcal{L}\bbS^2$ by the natural action $\mathrm{SO}(3)$ and that this gives an infinite-dimensional version of the results of \cite{kapovich} on polygon space $\mathrm{Pol}(n,\vec{r})$. In this section we will show that  $\mathrm{IsoImm}(S^1,\R^3)$ is realized as the symplectic reduction of $\mathcal{M}$ by the  action of $\mathcal{L}\bbS^1/\bbS^1$. This gives a smooth version of the main result of  \cite{hausmann}, which says that the polygon space $\mathrm{Pol}(n,\vec{r})$ may be realized as the symplectic reduction of the Grassmannian $\mathrm{Gr}_2(\C^n)$ by the action of $\mathrm{U}(1)^n/\mathrm{U}(1)$.

\subsection{The Actions of $\mathcal{P}\bbS^1$ and $\mathcal{L}\bbS^1$}\label{sec:loop_group_action}

Recall that $\bbS^1$ acts on framed path space $\op$ by constant frame twisting. We extend this idea to  define an action of the path group $\mathcal{P}\bbS^1$ on $\op$ by \emph{nonconstant} frame twisting. For $e^{i \alpha} \in \mathcal{P}\bbS^1$ and $(\gamma,V) \in \op$, the action is given explicitly by
\begin{equation}\label{eqn:frame_twisting_action}
(e^{i\alpha} \cdot (\gamma,V))(t)=(\gamma(t),\cos(2\alpha(t))V(t) + \sin(2\alpha(t))W(t)), \;\;\; W=\frac{\gamma'}{\|\gamma'\|} \times V.
\end{equation}
The action of $\mathcal{P}\bbS^1$ is transitive on the set of framings of a fixed path $\gamma$. Moreover,   Lemma \ref{lem:frame_hopf_equivariance} implies that the equivariance property 
$$
\widehat{\mathrm{H}}(e^{i\alpha} \Phi) = e^{i\alpha} \cdot \widehat{\mathrm{H}}(\Phi);
$$
holds for $e^{i\alpha} \in \mathcal{P}\bbS^1$ and $\Phi \in \mathcal{P}\C^2$; that is, the frame twisting action is given in complex coordinates by pointwise multiplication.

The action of $\mathcal{P}\bbS^1$ on $\op$ restricts to give the frame twisting action of $\mathcal{L}\bbS^1$ on $\cl$. The factor of 2 that appears in formula (\ref{eqn:frame_twisting_action}) implies that the frame-twisting action preserves path components of $\mathcal{C}$. The action also restricts to $\mathcal{C}/\R^+$ and the complex version of this action is given by pointwise multiplication in the Stiefel manifold.

When passing to the quotient $\mathcal{M}$ we obtain a well-defined action of $\mathcal{L}\bbS^1$, since frame twisting commutes with rigid rotations. However, this action is not free. This is because framed loops are identified in $\mathcal{M}$ if they differ by a global frame twist---i.e., we have already taken the quotient by the subgroup $\bbS^1 \subset \mathcal{L}\bbS^1$ to obtain $\mathcal{M}$. We will thus consider the free action of the \emph{based loop group} $\mathcal{L}\bbS^1/\bbS^1$ on $\mathcal{M}$. Let $[e^{i\alpha}] \in \mathcal{L}\bbS^1/\bbS^1$ denote the equivalence class of $e^{i\alpha} \in \mathcal{L}\bbS^1$. The corresponding action on $[\Phi] \in \Gr$ is given by
$$
[e^{i\alpha}] \cdot [\Phi] = [e^{i\alpha} \Phi].
$$

\subsection{Hamiltonian Structure}\label{sec:hamiltonian_structure}

It is easy to show that the frame twisting action of $\ls$ on $\M$ is by isometries, particularly when the action is expressed in the Grassmannian formalism. We now wish to show that the action is also Hamiltonian. Using the fact that the exponential map $\mathcal{L}\R \rightarrow \mathcal{L}\bbS^1$ is given explicitly by $\alpha \mapsto e^{i\alpha}$, one is able to show that the induced vector field of  $\alpha \in \mathcal{L}\R$ on $\St$ is given by $X_\alpha |_\Phi = i \alpha \Phi$. For computations, the vector field induced by $[\alpha] \in \mathcal{L}\R/\R$ on $\Gr$  is represented by
$$
X_{[\alpha]}|_{[\Phi]} = \mathrm{proj}(i\alpha \Phi),
$$
where $\mathrm{proj}$ is orthogonal projection
\begin{equation}\label{eqn:induced_vector_field}
\mathrm{proj}:T_\Phi \St \rightarrow T_\Phi^{hor} \St \approx T_{[\Phi]} \Gr
\end{equation}
onto the codimension-4 horizontal space. The projection is necessary since $\mathcal{L}\bbS^1/\bbS^1$-orbits of elements of $\St$ are not $L^2$-orthogonal to $\mathrm{U}(2)$-orbits (although the orbits are transverse). Note that this is formula  depends on the choice of representation $\Phi \in [\Phi]$, but does not depend on the choice of representation $\alpha \in [\alpha]$. Indeed, for $c \in \R$, $X_{[\alpha + c]}|_{[\Phi]} = \mathrm{proj}(i \alpha \Phi + ci\Phi)$ and $i\Phi$ is tangent to the $\mathrm{U}(2)$-orbit of $\Phi$.

We define an inner product on $\mathcal{L}\R/\R$ by the formula
\begin{equation}\label{eqn:inner_product_on_loops_R}
\left<[\alpha],[\beta]\right>_{\mathcal{L}\R/\R} := \frac{1}{2}\int_0^2 \alpha \beta \; \mathrm{d}t - \frac{1}{4}\int_0^2 \alpha \; \mathrm{d}t \int_0^2 \beta \; \mathrm{d}t, \;\;\; \alpha,\,\beta \in \mathcal{L}\R
\end{equation}
and use this to injectively map $\mathcal{L}\R/\R$ into its dual space. We then define our candidate for the moment map for the $\mathcal{L}\bbS^1/\bbS^1$-action to be
\begin{align*}
\mu=\mu_{\mathcal{L}\bbS^1/\bbS^1}:\mathrm{Gr}_2(\mathcal{L}\C) &\rightarrow \mathcal{L}\R/\R \\
[\phi,\psi] &\mapsto [|\phi|^2+|\psi|^2].
\end{align*}
In terms of framed loops, the momentum map has a simple interpretation: if $\widehat{\mathrm{H}}([\Phi])=[\gamma,V]$ then $\mu([\Phi])=[\|\gamma'\|]$. 

For fixed $[\alpha] \in \ls$, let $f_{[\alpha]}:\Gr \rightarrow \R$ denote the map
$$
f_{[\alpha]}([\Phi]) = \left<\mu([\Phi]),[\alpha]\right>_{\mathcal{L}\R/\R}.
$$
Then we wish to show that 
\begin{equation}\label{eqn:moment_map_condition}
Df_{[\alpha]}([\Phi])(\delta \Phi) = \omega_{[\Phi]}^{L^2}(\delta \Phi, X_{[\alpha]}|_{[\Phi]})
\end{equation}
To perform the calculation, we lift $[\Phi] \in \Gr$ to $\Phi \in \St$ and consider $\delta \Phi \in T^{hor}_\Phi \St$. The derivative on the lefthand side of (\ref{eqn:moment_map_condition}) is then given by
$$
\left.\frac{d}{d\epsilon}\right|_{\epsilon=0} \frac{1}{2}\int_0^2 (|\phi + \epsilon \delta\phi|^2 + |\psi + \epsilon \delta \psi|^2)\alpha \; \mathrm{d}t - \frac{1}{2} \int_0^2 \alpha \; \mathrm{d}t = \int_0^2 \mathrm{Re}(\phi \overline{\delta \phi} + \psi \overline{\delta \psi})\alpha \; \mathrm{d}t.
$$
The righthand side of (\ref{eqn:moment_map_condition}) is equal to 
\begin{align*}
-\mathrm{Im}\left<\delta \Phi,  X_{[\alpha]}|_{[\Phi]} \right>_{L^2} &= -\mathrm{Im}\left<\delta \Phi, i\alpha \Phi \right>_{L^2} = \mathrm{Re} \left<\alpha \Phi, \delta \Phi\right>_{L^2}  =  \int_0^2 \mathrm{Re}\left(\phi \overline{\delta \phi} + \psi \overline{\delta \psi}\right)\alpha \; \mathrm{d}t.
\end{align*}
We have omitted the projection from the notation for $X_{[\alpha]}|_{[\Phi]}$, since the symplectic form vanishes along the $\mathrm{U}(2)$-vertical direction by construction.

\subsection{Principal Bundle Structure}\label{sec:principal_bundle}

Consider the \emph{moduli space of (unframed) loops},
$$
\mathcal{B} := \mathrm{Imm}(S^1,\R^3)/\mathrm{Sim}
$$
Using our usual conventions, we consider elements of $\mathcal{B}$ as $\mathrm{SO}(3)$-orbits of based immersions with fixed length $2$. We denote the $\mathrm{SO}(3)$-orbit of a based immersion $\gamma$ by $[\gamma]$. The moduli space $\mathcal{B}$ is a Fr\'{e}chet manifold: the space of based immersions is an open submanifold of the vector space of based loops, Hamilton's Implicit Function Theorem can be used to show that the fixed length subspace is a manifold of codimension-1, and one can construct smooth cross-sections to the $\mathrm{SO}(3)$-orbits by an argument similar to \cite[Lemma 1.6]{millson}.

The main result of this section is the following. 

\begin{thm}\label{thm:principal_bundle}
Each component of the moduli space $\mathcal{M}$ is an $\ls$-bundle over $\mathcal{B}$.
\end{thm}

We will focus on the component $\M_{ev}$---the proof can be translated to odd-linking framed curves via the diffeomorphism $\M_{ev} \approx \M_{od}$. The proof follows from a pair of lemmas, the first of which is proved by an elementary computation. It is a statement about the total twist functional. For a framed path $(\gamma,V)$, the \emph{twist rate} is defined as
$$
\mathrm{tw}(\gamma,V):=\left<\frac{d}{ds} V, W \right>, \;\;\;\; W=\frac{\gamma'}{\|\gamma'\|} \times V
$$
and the \emph{total twist} is 
$$
\mathrm{Tw}(\gamma,V) := \frac{1}{2\pi} \int_0^2 \mathrm{tw}(\gamma,V) \; \mathrm{d}s.
$$
These quantities are invariant under rigid motions and constant frame twists, so they are well-defined on equivalence classes $[\gamma,V]$.

\begin{lem}\label{lem:twist_rate_change}
Let $[\gamma,V] \in \M$ and $[e^{i\alpha}] \in \ls$. Then 
\begin{equation}\label{eqn:twist_formula}
\mathrm{tw}([e^{i\alpha}]\cdot [\gamma,V]) = -\frac{2 \alpha'}{\|\gamma'\|} + \mathrm{tw}(\gamma,V).
\end{equation}
It follows that 
$$
\mathrm{Tw}([e^{\i\alpha}] \cdot [\gamma,V]) = \mathrm{Tw}(\gamma,V) - \frac{1}{\pi}(\alpha(2)-\alpha(0)).
$$
\end{lem}

From the second statement of the lemma, we have
$$
\mathrm{Tw}(\gamma,V_1) \; \mathrm{mod} \, 2 = \mathrm{Tw}(\gamma, V_2) \; \mathrm{mod} \, 2
$$
for any even-linking framings $V_1$, $V_2$ of the same base curve $\gamma$. We are therefore able to assign an invariant $\mathrm{Tw}_2(\gamma)$ to any loop $\gamma$ via the formula $\mathrm{Tw}_2(\gamma):=\mathrm{Tw}(\gamma,V) \; \mathrm{mod} \, 2$, where $V$ is any choice of even-linking framing of $\gamma$.

\begin{lem}
Let $[\gamma,V] \in \M_{ev}$ and let $\alpha(\gamma,V) \in \mathcal{P}\R$ be defined by
\begin{equation}\label{eqn:alpha_k}
\alpha(\gamma,V)(t)=\frac{1}{2}\int_0^t \mathrm{tw}(\gamma,V) \; \mathrm{d}s - \frac{\pi \mathrm{Tw}_2(\gamma)}{2} \int_0^t \; \mathrm{d}s
\end{equation}
Then $e^{i\alpha(\gamma,V)} \in \mathcal{L}\bbS^1$ and  $[e^{i\alpha(\gamma,V)}] \cdot [\gamma,V]$ has constant twist rate equal to $\pi \mathrm{Tw}_2(\gamma)$ and total twist equal to $\mathrm{Tw}_2(\gamma)$. 
\end{lem}

\begin{proof}
To see that $e^{i \alpha(\gamma,V)}$ is a smooth loop, note that its derivative
$$
\alpha(\gamma,V)'(t)=\frac{1}{2}\mathrm{tw}(\gamma(t),V(t))\|\gamma'(t)\|-\frac{\pi}{2}\mathrm{Tw}_2(\gamma)\|\gamma'(t)\|
$$
is a smooth loop in $\R$, that $\alpha(\gamma,V)(0)=0$ and that
$$
\alpha(\gamma,V)(2)=\pi \mathrm{Tw}(\gamma,V) - \pi \left( \mathrm{Tw}(\gamma,V) \; \mathrm{mod} \, 2 \right) = \pi \cdot 2k
$$
for some $k \in \Z$. Formula (\ref{eqn:twist_formula}) shows that $\mathrm{tw}([e^{i\alpha(\gamma,V)(t)}] \cdot [\gamma(t),V(t)])$ is given by
$$
-\frac{2}{\|\gamma'(t)\|}\left(\frac{1}{2} \mathrm{tw}(\gamma(t),V(T)) \|\gamma'(t)\| - \frac{\pi}{2} \mathrm{Tw}_2(\gamma) \|\gamma'(t)\|\right) + \mathrm{tw}(\gamma(t),V(t)) = \pi \mathrm{Tw}_2(\gamma).
$$
It follows immediately that $\mathrm{Tw}([e^{i\alpha(\gamma,V)}] \cdot [\gamma,V]) = \mathrm{Tw}_2(\gamma)$.
\end{proof}

For any loop $\gamma$, we can assign a relative framing $V(\gamma)$ called the \emph{constant twist minimizing framing (CTMF)} which is characterized up to constant frame twists by having constant twist rate equal to $\pi \mathrm{Tw}_2(\gamma)$. The CTMF is given by $e^{i\alpha(\gamma,V)}(\gamma,V)$ where $V$ is an arbitrary framing. The framing $V(\gamma)$ is referred to as \emph{constant twist minimizing} for the following reason. If $\gamma$ is an embedded loop, then the White-Fuller-C\u{a}lug\u{a}reanu Theorem states that for any choice of framing $V$, $\mathrm{Tw}(\gamma,V) + \mathrm{Wr}(\gamma) = \mathrm{Lk}(\gamma,V)$, where $\mathrm{Wr}(\gamma)$ is the \emph{writhe} of $\gamma$ (a geometric invariant of $\gamma$---see \cite{dennis}). It follows that $\mathrm{Tw}_2(\gamma)$ is the minimum possible positive total twist of any framing of $\gamma$ of even linking number, and that $\pi \mathrm{Tw}_2(\gamma)$ is the minimium possible positive constant twist rate for such a framing.

Recall from Section \ref{sec:projective_spaces} that the Bishop framing of a path $\gamma$ is a relative framing defined by evolving an initial vector $V(0)$ along $\gamma$ with no intrinsic twisting; i.e., the Bishop framing is a solution of the ODE $\mathrm{tw}(\gamma,V)=0$. For a generic closed loop $\gamma$, the Bishop framing does not give a \emph{closed} relative framing. Those loops $\gamma$ which do admit a closed Bishop framing will play a special role in our discussion, since they are the curves at which $\mathrm{Tw}_2$ is discontinuous.  More precisely, the map $\mathrm{Tw}_2:\mathcal{B} \rightarrow [0,2)$ is well-defined but discontinuous at curves that admit framings with total twist equal to an even integer---this is simply because the mod 2 map $\R \rightarrow [0,2)$ is discontinuous at even integers. If $(\gamma,V)$ is a framed loop with $\mathrm{Tw}(\gamma,V)=2k$ for some integer $k$, then $e^{i \alpha(\gamma,V)} \cdot (\gamma,V)$ gives a closed Bishop framing of $\gamma$. 

Let $\mathcal{U} \subset \mathcal{B}$ denote the open subset containing orbits of loops which do not admit a closed Bishop framing and let $\widetilde{\mathcal{U}} \subset \mathcal{M}_{ev}$ denote the open subset containing $[\gamma,V]$ with $[\gamma] \in \mathcal{U}$ and $V$ any framing of even linking number. Then $\mathrm{Tw}_2:\mathcal{U} \rightarrow [0,2)$ is a smooth map and it follows that $[\gamma,V] \mapsto [e^{i\alpha(\gamma,V)}] \cdot [\gamma, V]$ is a smooth map from $\widetilde{\mathcal{U}}$ to itself, by the explicit formula for $\alpha(\gamma,V)$.

\begin{proof}[Proof of Theorem \ref{thm:principal_bundle}]
The obvious projection $\M_{ev} \rightarrow \mathcal{B}$ is the \emph{forget framing map} $[\gamma,V] \rightarrow [\gamma]$. The fibers of this projection are diffeomorphic to $\ls$, as $\ls$ acts transitively and freely on the even-linking-number relative framings of a fixed base curve $[\gamma]$. 

It remains to show that $\M_{ev}$ is locally diffeomorphic to $\mathcal{B} \times \ls$. This will be accomplished by constructing smooth sections on an open cover of $\mathcal{B}$. The map $\mathcal{U} \rightarrow \widetilde{\mathcal{U}}$ given by $[\gamma] \mapsto [\gamma,V(\gamma)]$ is a section, since $V(\gamma)$ is uniquely determined up to constant frame twists. Moreover, the section is smooth by the above discussion.

We wish to mimic this construction on another open subset of $\mathcal{B}$. Let $\mathcal{U}' \subset \mathcal{B}$ be the open subset which excludes orbits of loops $\gamma$ such that the Bishop framing $V$ of $\gamma$ satisfies $V(2)=-V(0)$ and let $\widetilde{\mathcal{U}}' \subset \mathcal{M}_{ev}$ denote its preimage with respect to the forget framing map. Consider the \emph{moduli space of anti-framed loops} $\mathcal{M}_{anti}$ consisting of $(\mathrm{Sim}_0 \times \bbS^1)$-orbits of framed paths $(\gamma,V)$ such that $\gamma$ is a smooth loop and $V$ satisfies $V(2)=-V(0)$. This space is diffeomorphic to $\M$ via the map $\M_{anti} \rightarrow \M$ induced by taking $(\gamma,V)$ to $e^{i\pi t/4} \cdot (\gamma,V)$. Let $\widetilde{\mathcal{U}}'_{anti} \subset \M_{anti}$ denote the preimage of $\widetilde{\mathcal{U}}' \subset \M_{ev}$ under this diffeomorphism ($[\gamma,V] \in \widetilde{\mathcal{U}}'_{anti}$ is said to have an \emph{even anti-framing}).  By arguments similar to those above, we can construct a smooth section $\mathcal{U}' \rightarrow \widetilde{\mathcal{U}}'_{anti}$ which assigns to each loop its even anti-framing of minimal possible positive constant twist rate. Composing with the diffeomorphism $\M_{anti} \rightarrow \M$ yields a smooth section $\mathcal{U}' \rightarrow \widetilde{\mathcal{U}}'$. Since $\mathcal{U} \cup \mathcal{U}'$ forms an open cover of $\mathcal{B}$, this completes the proof.
\end{proof}

This theorem has a corollary which follows trivially but is of practical interest. A well-studied problem in applied differential geometry is to algorithmically assign a framing or relative framing to a given parameterized space curve \cite{carroll, dennis, hanson2, wang}. This has applications to computer graphics, where one uses the framing to construct a tube around a given curve for visual clarity \cite{hanson2}, as well as animation, motion planning and camera tracking \cite{wang}. Desirable properties of a curve framing algorithm include: the algorithm should be invariant under ambient Euclidean similarities (it should depend only on the geometry of the curve), the framing should close if the curve is a loop (necessary in computer graphics if the tube is to be textured \cite{hanson2})  and the framing should vary continuously as the curve varies (necessary for animation applications). We therefore define a \emph{curve framing algorithm} to be a continuous section from a subset of $\mathcal{B}$ to $\mathcal{M}$. 

Examples of curve framing algorithms include: the relative framing induced by the Frenet framing (defined on the set of loops with nonvanishing curvature), the Bishop framing (defined on the set of loops with integral writhe), and the Writhe framing (defined on the set of embedded curves). Each of these algorithms fails on some subset of curves. The following corollary says that this must be the case for any curve framing algorithm.

\begin{cor}\label{cor:framing_algorithm}
There is no curve framing algorithm defined on all of $\mathcal{B}$.
\end{cor}

\begin{proof}
We have shown that each component of $\mathcal{M}$ is a principal bundle over $\mathcal{B}$. A continuous global section would imply that one of the components of $\mathcal{M}$ is homeomorphic to $\mathcal{B} \times \ls$, which is not the case since $\ls$ has infinitely many path-components.
\end{proof}

\subsection{Reduction onto $\mathrm{IsoImm}(S^1,\R^3)$}

Let $\textbf{1} \in \mathcal{L}\R$ denote the constant loop in $\R$. The next step in the process of showing that $\mathrm{IsoImm}(S^1,\R^3)$ is the symplectic reduction of each component of $\M$ is to show that the level set $\mu^{-1}([\textbf{1}])$ is a manifold, where $\mu=\mu_{\ls}$ continues to denote the moment map of the $\ls$-action. We will use the notation
$$
\cl_1 := \{(\gamma,V) \in \cl \mid \|\gamma'(t)\|=1 \; \forall \; t\} \; \mbox{ and } \; \M_1:= \cl_1/(\mathrm{Sim}_0 \times \bbS^1).
$$
It follows from Lemma \ref{lem:frame_hopf_map} that $\widehat{\mathrm{H}}\left(\mu^{-1}([\textbf{1}])\right) = \mathcal{M}_1$, so it suffices to show that $\M_1 \subset \M$ is a submanifold.

\begin{prop}\label{prop:arclength_parameterized_submanifolds}
The spaces $\cl_1 \subset \cl$ and $\M_1 \subset \M$ are submanifolds.
\end{prop}

\begin{proof}
The embedding $\mathrm{SO}(3) \times \{\textbf{1}\} \hookrightarrow \mathrm{Sim}_0=\mathrm{SO}(3) \times \R^+$ induces an embedding of $\mathcal{L}(\mathrm{SO}(3) \times \{\textbf{1}\})$ into $\LSO$ as a submanifold, so it suffices to show that the image of $\cl_1$ under the frame map is a submanifold of $\mathcal{L}(\mathrm{SO}(3) \times \{\textbf{1}\})$. The image of $\cl_1$ is the set
$$
\left\{\left((U,V,W),\textbf{1}\right) \mid \int_{S^1} U \; \mathrm{d}t = \vec{0} \right\} \subset \mathcal{L}(\mathrm{SO}(3) \times \{\textbf{1}\}).
$$
That this set is a submanifold follows by applying \cite[Section~III, Theorem~2.3.1]{hamilton}, which is an extension of the implicit function theorem to maps from a tame Fr\'{e}chet space to a finite-dimensional vector space. To apply the theorem, we need to show that $\vec{0}$ is a regular value of the map $f:\mathcal{L}(\mathrm{SO}(3) \times \{\textbf{1}\}) \rightarrow \R^3$ defined by
$$
f:((U,V,W),\textbf{1}) \mapsto \int_{S^1}  U \; \mathrm{d}t.
$$
The tangent spaces to  $\mathcal{L}(\mathrm{SO}(3) \times \{\textbf{1}\})$ are isomorphic to $\mathcal{L}(\mathfrak{so}(3) \times \{\textbf{0}\})$, where $\textbf{0}$ denotes the constant zero map. We can therefore express a tangent vector at $\left((U,V,W),\textbf{1}\right) \in \mathcal{L}(\mathrm{SO}(3) \times \{\textbf{1}\})$ as a variation $\left((\delta U, \delta V, \delta W),\textbf{0}\right)$, with $\delta U = \xi^1 V - \xi^2 W$ for some $\xi^1, \xi^2 \in \mathcal{L}\R$. Then the derivative of $f$ is given by
$$
Df((U,V,W),r)((\delta U, \delta V, \delta W),0) = \int_{S^1} \delta U \; \mathrm{d}t = \int_{S^1} \xi^1 V - \xi^2 W \; \mathrm{d}t.
$$

To show that $\cl_1$ is a submanifold, we wish to show that
\begin{equation}\label{eqn:full_span}
\mathrm{span}\left\{\int_{S^1} \xi^1 V - \xi^2 W \; \mathrm{d}t \mid \xi^1, \xi^2 \in \mathcal{L}\R \right\}=\R^3.
\end{equation}
Toward this goal, we claim that there exists $t_0 \in S^1$ such that 
$$
\mathrm{span}\{V(0),W(0),V(t_0),W(t_0)\} = \R^3.
$$
Indeed, the span of $V(0)$ and $W(0)$ is already 2-dimensional and if $V(t),W(t) \in \mathrm{span}\{V(0),W(0)\}$ for all $t$ then $U(t)$ must be constant, and this contradicts the assumption that $f((U,V,W),\textbf{1})=\vec{0}$. Assuming without loss of generality that $V(t_0)$ is linearly independent of $V(0)$ and $W(0)$, we choose bump functions $\xi^1_1$ and $\xi^2_2$ around $0$ and $\xi^1_3$ around $t_0$ with sufficiently small support so that
$$
\int_{S^1}  \xi^1_1 V \; \mathrm{d}t, \; \int_{S^1}  \xi^2_2 W \; \mathrm{d}t, \; \mbox{ and } \; \int_{S^1} \xi^1_3 V \; \mathrm{d}t
$$
are linearly independent. These vectors belong to the left hand side of (\ref{eqn:full_span}) and this shows that $\cl_1$ is a submanifold of $\cl$. Since this construction is $\mathrm{Sim}_0 \times \bbS^1$-invariant, the same approach can be used to show that $\M_1$ is a submanifold.
\end{proof}

We now wish to characterize the horizontal and vertical tangent directions with respect to the $\ls$-action on $\M$. It will be useful to work in complex coordinates. For simplicity, we focus on $\mathrm{Gr}_2(\mathcal{L}\C)$---all of the statements can be translated to $\mathrm{Gr}_2(\mathcal{A}\C)$. Recall we have used the notation for $T^{hor}_\Phi \mathrm{St}_2(\mathcal{L}\C)$ for the $\mathrm{U}(2)$-horizontal subspace and identified $T_{[\Phi]} \mathrm{Gr}_2(\mathcal{L}\C) \approx T_{[\Phi]}^{hor} \mathrm{St}_2(\mathcal{L}\C)$. We use $\mathrm{proj}$ to denote orthogonal projection $T_\Phi \mathrm{St}_2(\mathcal{L}\C) \rightarrow T^{hor}_\Phi \mathrm{St}_2(\mathcal{L}\C)$.

\begin{lem}\label{lem:splitting}
The tangent space to $\mu^{-1}([\textbf{1}])$ splits orthogonally as
$$
T_{[\Phi]} \mu^{-1}([\textbf{1}]) = T_{[\Phi]}^{vert} \mu^{-1}([\textbf{1}]) \oplus T_{[\Phi]}^{hor} \mu^{-1}([\textbf{1}]),
$$
where $T_{[\Phi]}^{vert} \mu^{-1}([\textbf{1}])$ is the $\ls$-vertical tangent space and
$$
T_{[\Phi]}^{hor} \mu^{-1}([\textbf{1}]) = \{ \delta \Phi \in T^{hor}_\Phi \mathrm{St}_2(\mathcal{L}\C) \mid \left<\Phi, \delta \Phi\right>_{\C^2} = 0\}
$$
is the complex subspace of $\ls$-horizontal tangents.
\end{lem}

\begin{proof}
Lift $[\Phi] \in \mu^{-1}([\textbf{1}])$ to $\Phi \in \mathrm{St}_2(\mathcal{L}\C)$ and let $\delta \Phi \in T^{hor}_\Phi \mathrm{St}_2(\mathcal{L}\C)$. Then 
$$
D\mu([\Phi])(\delta \Phi) = \left.\frac{d}{d\epsilon}\right|_{\epsilon=0} |\phi + \epsilon \delta \phi |^2 + |\psi + \epsilon \delta \psi |^2 = 2 \mathrm{Re}(\phi \overline{\delta \phi} + \psi \overline{\delta \psi}) = \mathrm{Re}\left<\Phi,\delta \Phi\right>_{\C^2}
$$
and 
$$
T_{[\Phi]} \mu^{-1}([\textbf{1}]) = \mathrm{ker} D\mu([\Phi])(\delta \Phi) \approx \{\delta \Phi \in T^{hor}_\Phi \mathrm{St}_2(\mathcal{L}\C) \mid \mathrm{Re} \left<\Phi, \delta \Phi \right>_{\C^2} = 0\}.
$$
The $\ls$-vertical directions of $\mathrm{Gr}_2(\mathcal{L}\C)$ were already described in Section \ref{sec:hamiltonian_structure}, whence we conclude
$$
T_{[\Phi]}^{vert} \mu^{-1}([\textbf{1}]) = \{ \mathrm{proj}(i\alpha \Phi) \mid \alpha \in \mathcal{L}\R\},
$$
where $\mathrm{proj}:T_\Phi \lSt \rightarrow T_\Phi^{hor} \lSt$ denotes orthogonal projection. The $\ls$-horizontal tangent space to $\mu^{-1}([\textbf{1}])$ is given by 
$$
\{\delta \Phi \in T^{hor}_{\Phi} \lSt \mid \mathrm{Re} \left<\Phi, \delta \Phi \right>_{\C^2} = 0 \mbox{ and } \mathrm{Re}\left<i\alpha \Phi , \delta \Phi \right>_{L^2} = 0 \mbox{ for all } \alpha \in \mathcal{L}\R \}.
$$
The second defining condition can be rewritten as
$$
\mathrm{Re}\left<i\alpha \Phi, \delta \Phi\right>_{L^2} =-\mathrm{Im}\left<\alpha \Phi, \delta \Phi \right>_{L^2} = -\int_{S^1} \alpha \mathrm{Im}\left<\Phi,\delta \Phi \right>_{\C^2} \; \mathrm{d}t = 0
$$
for all $\alpha \in \mathcal{L}\R$, and we deduce that $\mathrm{Im}\left<\Phi,\delta\Phi\right>_{\C^2} = 0$.

From these characterizations, it is easy to see that the intersection of the horizontal and vertical spaces is zero. To see that the tangent space splits orthogonally, we note that the orthogonal projection operator $T_{[\Phi]} \mu^{-1}([\textbf{1}]) \rightarrow T_{[\Phi]}^{vert} \mu^{-1}([\textbf{1}])$  is given explicitly by $\mathrm{proj}\left(\left< \Phi, \delta \Phi\right>_{\C^2} \Phi\right)$.
\end{proof}

We now arrive at the main result of this section. 

\begin{thm}\label{thm:millson_zombro_space}
The moduli space of isometric immersions $\mathrm{IsoImm}(S^1,\R^3)$ is obtained as a symplectic reduction of either component of $\M$ by the action of $\ls$. The induced symplectic structure agrees with the Millson-Zombro symplectic form up to a constant.
\end{thm}

\begin{proof}
Going through the proof of Theorem \ref{thm:principal_bundle}, we see that it can be directly adapted to show that each component of $\M_1$ is an $\ls$-bundle over $\mathrm{IsoImm}(S^1,\R^3)$. It follows that each component of $\M_1/(\ls)$ is diffeomorphic to $\mathrm{IsoImm}(S^1,\R^3)$. It follows from Lemma \ref{lem:splitting} that $\M \sslash (\ls) = \M_1/(\ls)$ inherits a well-defined Riemannian metric, symplectic form and almost complex structure and that these structures are compatible. 

To see that the induced structures agree with those of Millson-Zombro, we recall from Section \ref{sec:induced_geometric_structures} that the admissible variations of $q \in \mathcal{P}\mathbb{H} \approx \mathcal{P}\mathbb{C}^2$ can be written in the form
$$
\lambda_1 \frac{q}{2} + \lambda_2 \frac{\textbf{i} q}{2} + \lambda_3 \frac{\textbf{j}q}{2} + \lambda_4 \frac{\textbf{k}q}{2},
$$
where $\lambda_j \in \mathcal{P}\R$. We also saw that the admissible variations of $(\gamma,V) \in \op$ can be written as combinations
$$
\lambda_{st} X_{st} + \lambda_{tw}X_{tw} + \lambda_{b_1} X_{b_1} + \lambda_{b_2} X_{b_2}
$$
of the basic variations and that $D\widehat{\mathrm{H}}$ gives a correspondence
$$
q/2 \mapsto X_{st}, \; \qi q/2  \mapsto X_{tw}, \; \qj q/2 \mapsto X_{b_1}, \; \qk q/2 \mapsto X_{b_2}.
$$
The condition $\left<\Phi,\delta \Phi \right>_{\C^2}=0$ defining $T_{[\Phi]}^{hor} \mu^{-1}([\textbf{1}])$ implies that the $\ls$-horizontal variations of an element of $\mu^{-1}([\textbf{1}])$ must have $\lambda_1=\lambda_2=0$ in the quaternionic notation. This implies that an $\ls$-horizontal variation of an element of $\M_1$ must take the form $\lambda_{b_1} X_{b_1} + \lambda_{b_2} X_{b_2}$; that is, the variation cannot have any stretching or twisting component. This means that the  induced almost complex structure of $\mathcal{M} \sslash (\ls)$ can be succinctly rewritten as
$$
J^\op \cdot \delta \gamma = T \times \delta \gamma.
$$
This almost complex structure  agrees with the Millson-Zombro almost complex structure of $\mathrm{IsoImm}(S^1,\R^3)$. Moreover, the induced Riemannian metric reduces to 
$$
g^\op((\delta \gamma_1, \delta V_1),(\delta \gamma_2, \delta V_2)) = \frac{1}{4} \int_0^2 \left<\frac{d}{ds} \delta \gamma_1, \frac{d}{ds} \delta \gamma_2 \right> \; \mathrm{d}s =  \frac{1}{4} \int_0^2 \left< \delta \gamma_1', \delta \gamma_2' \right> \; \mathrm{d}t
$$
and this agrees with the Millson-Zombro metric up to a constant multiple of $1/4$.  Therefore the symplectic structure agrees up to a constant as well.
\end{proof}

\section{The Reparameterization Action}\label{sec:reparameterization_action}

\subsection{The $\mathrm{Diff}^+(S^1)$-Action}\label{sec:diff_plus_action}

As previously mentioned, the group $\mathrm{Diff}^+(S^1)$ of orientation-preserving diffeomorphisms of $S^1$ acts on $\M$ by reparameterizations. This action is important for shape recognition applications, where one wishes to do computations in the \emph{space of unparameterized shapes} $\mathcal{M}/\mathrm{Diff}^+(S^1)$ (see Section \ref{subsubsec:elastic_shape_analysis}). In this section we study the symplectic geometry of the $\mathrm{Diff}^+(S^1)$-action.

One should immediately notice that this action is not well-defined with respect to our usual conventions; i.e., reparameterization does not preserve basepoints, so representing elements of $\M$ as framed loops based at the origin is no longer a sensible option. When dealing with the $\mathrm{Diff}^+(S^1)$-action we will denote elements of $\M$ by $\llbracket \gamma , V \rrbracket$, where $(\gamma, V)$ is a closed framed loop of length 2, not necessarily based at the origin, and 
$$
\llbracket \gamma, V \rrbracket := [\gamma - \gamma(0), V-V(0)].
$$  
Here $[\cdot, \cdot]$ continues to denote the equivalence class of a based framed loop under the actions of $\mathrm{SO}(3)$ by rotation and $\bbS^1$ by constant frame twisting. Then the action of $\rho \in \mathrm{Diff}^+(S^1)$ on $\llbracket \gamma, V \rrbracket \in \M$ is given by
$$
\rho \cdot \llbracket \gamma, V \rrbracket = \llbracket \gamma \circ \rho, V \circ \rho \rrbracket.
$$

The diffeomorphism from $\M$ to $\mathrm{Gr}_2^\circ(\mathcal{L}\C) \sqcup \mathrm{Gr}_2^\circ(\mathcal{A}\C)$ involves taking a derivative, so the issue with basepoint preservation is not relevant in Grassmannian coordinates. An element $\rho \in \Dif$ acts on $[\Phi] \in \Gr$ by the formula
$$
\rho \cdot [\Phi] = [\sqrt{\rho'} \cdot \Phi \circ \rho].
$$
We leave it to the reader to check that $\widehat{H}$ is equivariant with respect to these actions of $\Dif$; that is, if $[\Phi] \in \Gr$ maps to $[\gamma,V]$ then
$$
[\sqrt{\rho'} \cdot \Phi \circ \rho] = \llbracket \gamma \circ \rho, V \circ \rho \rrbracket.
$$

\subsection{Hamiltonian Structure}\label{sec:diff_plus_action}

Our next goal is to show that the reparameterization action is Hamiltonian. It will be convenient to work in complex coordinates and we restrict our attention to $\lGr$---the same arguments work for the anti-loop Grassmannian. 

The construction of the momentum map for the action of $\Dif$ is similar to the construction for $\mathcal{L}\bbS^1/\bbS^1$ in Section \ref{sec:hamiltonian_structure} so we will skip some details. We begin by determining a formula for the vector field induced by $\xi \in \mathcal{L}\R$.  Let $\rho_\epsilon$ be a path in $\mathrm{Diff}^+(S^1)$ with $\rho_0$ the identity and let $\left.\frac{d}{d\epsilon}\right|_{u=0} \rho_\epsilon = \xi \in \mathcal{L}\R$. Then
$$
\left.\frac{d}{d\epsilon}\right|_{\epsilon=0} \Phi(\rho_\epsilon) \sqrt{\rho'_\epsilon} = \Phi'(\rho_0) \xi \sqrt{\rho'_0} + \Phi (\rho_0) \frac{1}{2}(\rho'_0)^{-1/2} \xi' = \frac{1}{2}\xi' \Phi + \xi \Phi'.
$$
We conclude that the the vector field induced by $\xi \in \mathrm{Diff}^+(S^1)$ on $\lGr$ is represented by
$$
X_{\xi}|_{[\Phi]} = \mathrm{proj}\left(\frac{1}{2}\xi' \Phi + \xi \Phi'\right),
$$
where $\mathrm{proj}$ denotes orthogonal projection onto the horizontal tangent space of $\lSt$ at $\Phi$---the $\Dif$-orbits of $\lSt$ are not, in general, $L^2$-orthogonal to the $\mathrm{U}(2)$-orbits. Note that (as was the case with the vector fields induced by the $\ls$-action on $\lGr$) this representation depends on the choice of lift $\Phi \in \lSt$ of $[\Phi] \in \lGr$.

We endow $\mathcal{L}\R$ (the Lie algebra of $\Dif$) with the $L^2$ metric
$$
\left<\xi_1,\xi_2\right>_{\mathcal{L}\R} := \frac{1}{2} \int_{S^1} \xi_1 \xi_2 \; \mathrm{d}t
$$
in order to embed it into its dual space. Our proposed momentum map is
\begin{align}
\mu_\dif:\M &\rightarrow \mathcal{L}\R \nonumber \\
[\Phi] &\mapsto  \mathrm{Im}(\phi' \overline{\phi} + \psi' \overline{\psi}) \label{eqn:diff_momentum_map}.
\end{align}
One can show that if $\Phi \in \mathcal{P}\C^2$ maps to $(\gamma,V)$ under $\widehat{\mathrm{H}}$, then
$$
\mathrm{tw}(\gamma,V)=\frac{2\mathrm{Im}\left<\Phi,\Phi'\right>_{\C^2}}{\|\Phi\|_{\C^2}^4}.
$$
It follows from this formula that the map $\mu_\dif$ has a natural interpretation in terms of framed curves: if $[\Phi]$ maps to a framed loop $[\gamma,V ]$ under $\widehat{\mathrm{H}}$, then the image of $\mu_\dif$ is $-\frac{1}{2}\mathrm{tw}(\gamma,V)\|\gamma'\|^2$.

Our goal is to show
\begin{equation}\label{eqn:diff_momentum_map_2}
Df_{\xi}([\Phi])(\delta \Phi)= \omega^{L^2}_{[\Phi]}(\delta \Phi, X_{\xi}|_{[\Phi]})
\end{equation}
for $\delta \Phi \in T_\Phi^{hor} \lSt$, where $f_{\xi}([\Phi]):=\left<\mu_\mathrm{Diff}([\Phi]),X_{\xi}|_{[\Phi]}\right>_{\mathcal{L}\R}$ for fixed $\xi \in \mathcal{L}\R$. A calculation similar to the one in Section \ref{sec:hamiltonian_structure} shows that the left hand side of (\ref{eqn:diff_momentum_map_2}) is
\begin{equation}\label{eqn:momentum_map_LHS}
\frac{1}{2}\int_{S^1} \xi \mathrm{Im}(\phi' \overline{\delta \phi} + \delta \phi' \overline{\phi} + \psi' \overline{\delta \psi} + \delta \psi' \overline{\psi}) \; \mathrm{d}t
\end{equation}
and integration by parts shows that the right hand side is
\begin{align*}
-\mathrm{Im} \int_{S^1} \frac{1}{2} \xi (\delta \phi \overline{\phi}' + \delta \psi \overline{\psi}') + \frac{1}{2}\xi'(\delta \phi \overline{\phi} + \delta \psi \overline{\psi}) \; \mathrm{d}t &= \frac{1}{2} \int_{S^1} \xi \mathrm{Im}(\delta \phi ' \overline{\psi} - \delta \psi \overline{\psi}' + \delta \psi' \overline{\psi} - \delta \psi \overline{\psi}') \; \mathrm{d}t\\
&= \frac{1}{2}\int_{S^1} \xi \mathrm{Im}(\phi' \overline{\delta \phi} + \delta \phi' \overline{\phi} + \psi' \overline{\delta \psi} + \delta \psi' \overline{\psi}) \; \mathrm{d}t.
\end{align*}

\subsection{The Basepoint Action and Weighted Total Twist}\label{sec:basepoint_action}

We define the \emph{weighted total twist functional} $\widetilde{\mathrm{Tw}}$ on $\mathcal{M}$ by 
$$
\widetilde{\mathrm{Tw}}(\llbracket \gamma,V \rrbracket):=\frac{1}{2\pi} \int_{S^1} \mathrm{tw}(\gamma,V) \|\gamma'\| \; \mathrm{d}s.
$$
This functional is similar to the classical total twist $\mathrm{Tw}$ but has an added weight of $\|\gamma'\|$ in the integrand. The weighting makes $\widetilde{\mathrm{Tw}}$ a more interesting functional on the space of \emph{parameterized} framed loops $\M$ than the unweighted $\mathrm{Tw}$ functional. In this section we characterize its critical points.

Consider the action of the subgroup $S^1 \subset \Dif$ consisting of pure rotations. We will refer to this $S^1$-action as the \emph{basepoint action}, as it can be interpreted as changing the basepoint of a based framed loop. The next lemma shows that the restricted action of $S^1$ on $\M$ is Hamiltonian with momentum map a constant multiple of $\widetilde{\mathrm{Tw}}$. The proof follows by a calculation similar to  the one in the previous section.

\begin{lem}\label{lem:basepoint_moment_map}
The basepoint action of $S^1$ on $\Gr$ is Hamiltonian with momentum map
\begin{align*}
\mu_{S^1}:\Gr &\rightarrow \R \approx \R^\ast \\
[\Phi] &\mapsto \int_{S^1} \mathrm{Im}(\phi' \overline{\phi} + \psi' \overline{\psi}) \; \mathrm{d}t.
\end{align*}
For a framed loop $[\gamma,V]$ this map has the form
$$
\mu_{S^1}(\llbracket \gamma,V \rrbracket)=-\frac{1}{2} \int_{S^1} \mathrm{tw}(\gamma,V)\|\gamma'\|^2 \; \mathrm{d}t.
$$
\end{lem}

\begin{thm}\label{thm:twist_critical_points}
The critical points of $\widetilde{\mathrm{Tw}}:\mathcal{M}\rightarrow \R$ are equivalence classes $\llbracket \gamma,V \rrbracket$ such that $\gamma$ is an arclength parameterized, length-2, multiply covered round circle and $V$ has constant twist rate. In complex coordinates, the critical points are represented by symmetric torus knots on the Clifford torus in $\bbS^3$.
\end{thm}

\begin{proof}
Lemma \ref{lem:basepoint_moment_map} implies that the critical points of $\mu_{S^1}$ (and hence of $\widetilde{\mathrm{Tw}}$) are exactly the fixed points of the basepoint $S^1$-action. For $u \in S^1 \approx [0,2]/(0\sim 2)$, we abuse notation slightly and write the basepoint action as $\llbracket \gamma(t),V(t) \rrbracket \mapsto \llbracket \gamma(t+u),V(t+u) \rrbracket$.  Then $\llbracket \gamma,V \rrbracket$ is fixed under the basepoint action if and only if $\llbracket \gamma(t+u),V(t+u)\rrbracket =\llbracket \gamma(t),V(t)\rrbracket$ for all $t$ and $u$. Unraveling the notation, this means that for each $u \in S^1$ there exists $e^{i\theta(u)} \in \bbS^1$ and $A(u) \in \mathrm{SO}(3)$ such that 
\begin{equation}\label{eqn:total_twist}
(\gamma(t+u)-\gamma(u),V(t+u)-V(u)) =e^{i\theta(u)} \cdot A(u) \cdot \left[(\gamma(t)-\gamma(0),V(t)-V(0)) \right]
\end{equation}
holds for all $t$, where $e^{i\theta(u)}$ acts by a constant frame twist and $A(u)$ acts by a rigid rotation. In particular,
$$
\gamma(t+u) = A(u) \gamma(t) + (\gamma(u) - A(u) \cdot \gamma(0)).
$$
Taking the norm of the $t$-derivative of this expression yields
$$
\left\|\frac{d}{dt} \gamma(t+u)\right\| = \left\| \frac{d}{dt} A(u) \gamma(t)\right\| = \|\gamma'(t)\|
$$
for all $u$ and $t$, and we conclude that $\gamma$ must have constant parameterization speed. Since $\gamma$ has length 2, it must be that $\gamma$ is arclength-parameterized. Similarly, the invariance of curvature under rigid motions implies that $\gamma$ has constant curvature. The fact that $\gamma$ is a closed loop implies that its constant curvature must be positive, so its torsion is well-defined and the same argument can be applied to show that it must be constant as well. We conclude that $\gamma$ is a round, arclength-parameterized circle. The same type of argument shows that $(\gamma,V)$ must have constant twist rate.

For a critical $\llbracket \gamma,V \rrbracket$ we give an explicit complex representation as a knot on the Clifford torus in $\bbS^3 \subset \C^2$. Assume that $\gamma$ is an $h$-times-covered circle and that $\mathrm{Lk}(\gamma,V)=k$; that is, $V$ has constant twist rate $\pi k$. We claim that the Clifford torus knot
$$
\Phi_{h,k}(t):=\frac{1}{\sqrt{2}}\left(\exp\left(\frac{i}{2}(k+h)\pi t\right),\exp\left(\frac{i}{2}(k-h)\pi t\right)\right)
$$
maps to $\llbracket \gamma,V \rrbracket$ under $\widehat{\mathrm{H}}$. Indeed, applying $\widehat{\mathrm{H}}$ to $\Phi_{h,k}$ yields the $h$-times covered circle
$$
\gamma(t)=\frac{1}{\pi h} \left(0,-\cos(h \pi t), \sin (h \pi t)\right).
$$
Using the formula from Section \ref{sec:diff_plus_action}, it is straightforward to check that the twist rate of the framed curve $\widehat{\mathrm{H}}(\Phi_{h,k})$ is $\pi k$.
\end{proof}

\section{Riemannian Geometry of Framed Loop Space}\label{sec:riemannian_geometry}

\subsection{Explicit Geodesics in Framed Loop Space}\label{sec:explicit_geodesics}

The identifications of various moduli spaces of framed paths with classical manifolds given in Section \ref{sec:complex_coordinates} allow us to describe the geodesics of these moduli spaces quite concretely. We will focus on the geodesics of $\M$, which have the most interesting description. By Theorem \ref{thm:grassmannian}, the geodesics of the moduli space of framed loops $\mathcal{M}$ are locally the geodesics of $\Gr$ and can therefore be described explicitly. Let $[\Phi_0]$, $[\Phi_1] \in \mathrm{Gr}_2(\mathcal{V})$:
\begin{itemize}
\item[1.] Compute the singular value decomposition of the orthogonal projection map $[\Phi_0] \rightarrow [\Phi_1]$ (considering the points as complex 2-planes). This produces new orthonormal bases $\widetilde{\Phi}_0=\left(\tilde{\phi}_0,\tilde{\psi}_0\right)$ for $[\Phi_0]$ and $\widetilde{\Phi}_1=\left(\tilde{\phi}_1,\tilde{\psi}_1\right)$ for $[\Phi_1]$ such that the orthogonal projection map takes the form $\tilde{\phi}_0 \mapsto \lambda_\phi \tilde{\phi}_1$ and $\tilde{\psi}_0 \mapsto \lambda_\psi \tilde{\psi}_1$, where $0\leq \lambda_\phi, \lambda_\psi \leq 1$ are the singular values.

\item[2.] Let $\theta_\phi= \arccos (\lambda_\phi)$ and $\theta_\psi = \arccos (\lambda_\psi)$. These are the Jordan angles of $[\Phi_0]$ and $[\Phi_1]$.

\item[3.] Assuming generically that $\theta_\phi, \theta_\psi \neq 0$ (the formulas are easy to modify otherwise), the geodesic joining the subspaces is given by $[\Phi_u]$, where $\Phi_u=(\phi_u,\psi_u)$ is described by the formulas
\begin{align*}
\phi_u (t) &= \frac{\sin ((1-u) \theta_\phi) \tilde{\phi}_0(t) + \sin (u \theta_\phi) \tilde{\phi}_1(t)}{\sin \theta_\phi} \\
\psi_u (t) &= \frac{\sin ((1-u) \theta_\psi) \tilde{\psi}_0(t) + \sin (u \theta_\psi) \tilde{\psi}_1(t)}{\sin \theta_\psi}.
\end{align*}
\end{itemize}
The geodesic distance between $[\Phi_0]$ and $[\Phi_1]$ is $\sqrt{\theta_\phi^2 +\theta_\psi^2}$.

In \cite{younes} the authors showed that the space of similarity classes of immersed planar loops is locally isometric to a real infinite-dimensional Grassmannian (see Section \ref{subsubsec:elastic_shape_analysis} and Section \ref{sec:totally_geodesic_submanifolds} below) and they gave a similar description of geodesics in planar loop space. This description of Grassmannian geodesics is based on work of Neretin \cite{neretin}.

\subsection{Regularity Issues}

We note that the geodesics of $\Gr$ are only \emph{locally} the geodesics of $\M$, since a geodesic in $\Gr$ does not necessarily stay in the open subset $\mathrm{Gr}_2^\circ(\mathcal{V})$. The geodesic completion of $\mathrm{Gr}_2^\circ(\mathcal{V})$ is $\Gr$ and the geodesics in the full Grassmannian correspond to framed curve evolutions which pass through singular framed curves with nonimmersed points and degenerate framings. One benefit of this is that the completion $\mathrm{Gr}_2(\mathcal{L}\C) \cup \mathrm{Gr}_2(\mathcal{A}\C)$ is connected, but there are many technical issues which need to be considered here. Foremost is that one of the main goals that one has when using this framework for shape recognition applications is to compute geodesics in the space of unparameterized shapes $\M/\Dif$ by optimizing geodesic distance over $\Dif$-orbits. Orbits of the induced $\mathrm{Diff}^+(S^1)$-action on the full Grassmannian are not closed, whence the quotient $\Gr/\mathrm{Diff}^+(S^1)$ is not Hausdorff! This suggests that one should further take the metric completion of the Grassmannian and an appropriate completion of $\Dif$. These technical issues are beyond the scope of this paper and will be treated in future work \cite{needham2}. The rest of this paper will only require the use of sufficiently short geodesics so that lack of completeness will not cause any problems.

\subsection{Examples}

In Figure \ref{fig:geodesics} we give two examples of geodesics in $\M$ using this formula. The geodesic in the top row is between a trefoil with its standard torus knot parameterization and a round, arclength-parameterized circle, each with (the relative framings induced by) their Frenet framings. We denote the these endpoints by $(\gamma_1,V_1)$ and $(\gamma_2,V_2)$, respectively. Framed curves are represented as thickened base curves with a line along their surfaces representing the twisting of the frame. That this geodesic is in the moduli space $\M$ means that the starting and ending positions are optimally aligned over $\mathrm{SO}(3)$ and constant frame twists, and each curve throughout the homotopy is length 2 and based at the origin. Note that this geodesic goes through a singular framed curve with cusps, illustrating the fact that $\M$ is not geodesically closed. The geodesic distance between these framed curves is $\approx 0.71$, where we have normalized the Grassmannian to have diameter 2.

The geodesic in the second row also starts at a standard parameterization of a trefoil with its Frenet framing. It also ends at a circle, but the parameterization of this circle has been chosen to minimize geodesic distance between the fibers of the $\mathrm{Diff}^+(S^1)$-action. That is, the circle has been reparameterized by approximating a realization of
$$
\inf_{\rho \in \diff} \mathrm{dist}((\gamma_1,V_1),(\gamma_2 \circ \rho, V_2 \circ \rho)),
$$
where $\mathrm{dist}$ is geodesic distance in $\M$. This is the standard elastic shape analysis approach, as outlined in Section \ref{subsubsec:elastic_shape_analysis}. The infimum is approximated using a dynamic programming algorithm. The framing of the circle is then chosen to minimize the distance between the $\ls$-fibers of the framed curves. The evolution of the base curve can therefore be seen as a geodesic in the space $\mathcal{B}/\mathrm{Diff}^+(S^1)$ of unparameterized, unframed loops. The evolution is much less symmetric in this case. Also note that the resulting framing of the circle has twist rate zero outside of 3 small regions where all of the twisting is localized. The geodesic distance between these framed curves is $\approx 0.46$.

The details of the algorithms used to lift the curves to the Grassmannian in order to calculate the geodesic, to approximate the optimal parameterization and to find the optimal framing will be discussed in future work \cite{needham2}. We also plan to give more concrete applications of this framework to elastic shape analysis of shapes such as protein backbones, DNA minicircles and level curves on surfaces.

\begin{figure}\label{fig:geodesics}
\includegraphics[scale=0.25]{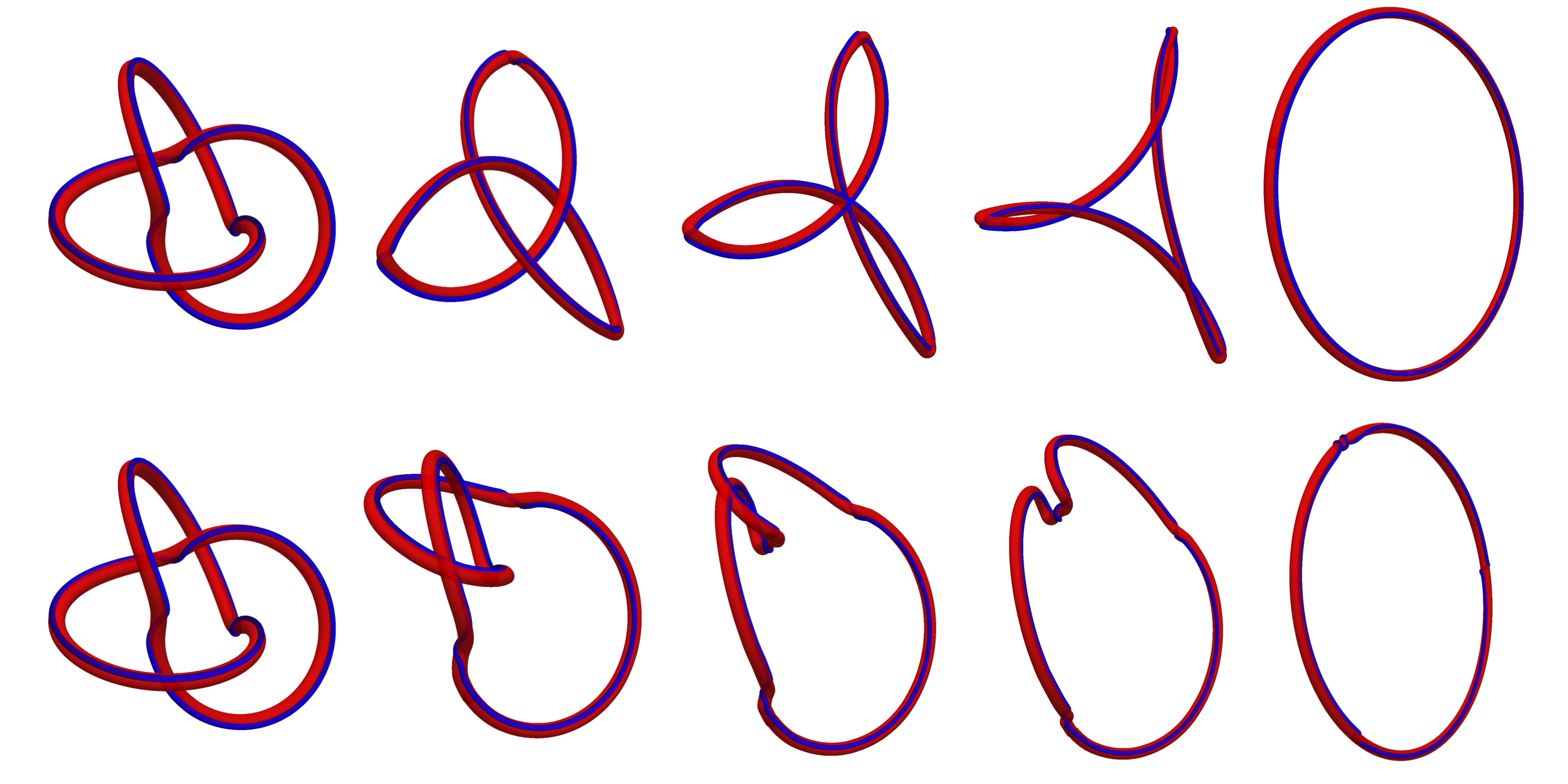}
\caption{Geodesics in $\M$; see text for details.}
\end{figure}

\subsection{The Exponential Map}

We can also obtain explicit formulas for the exponential maps in $\St$ and $\Gr$. The derivation is essentially the same as the derivation given by Edelman et.\ al.\  in \cite{edelman} for finite-dimensional real Stiefel manifolds. In the following proposition we identify $\Phi=(\phi,\psi) \in \mathcal{V}^2$ with the linear map $\C^{2 \times 2} \rightarrow \mathcal{V}^2$ defined by pointwise left multiplication. We use the notation $\Phi^\ast:\mathcal{V}^2 \rightarrow \C^{2 \times 2}$ for the \emph{formal adjoint} of the map $\Phi$, defined by
$$
\Phi^\ast (\phi_1,\psi_1)=\left(\begin{array}{cc}
\left<\phi_1,\phi \right>_{L^2} & \left<\psi_1,\phi \right>_{L^2} \\
\left<\phi_1,\psi \right>_{L^2} & \left<\psi_1,\psi \right>_{L^2}\end{array}\right).
$$
We are now able to state the geodesic equation and its solution for $\St$. In the proposition, paths in $\St$ are written as $\Phi_u$, where $u$ is the homotopy parameter. 

\begin{prop}\label{prop:geodesic_equation_stiefel}
The geodesic equation for $\St$ with respect to $g^{L^2}$ is
$$
\frac{\partial^2}{\partial u^2} \Phi_u + \Phi_u \left(\frac{\partial}{\partial u} \Phi_u \right) ^\ast \frac{\partial}{\partial u} \Phi_u = 0.
$$
The geodesic $\Phi_u$ with $\Phi_0=\Phi$ and $\left.\frac{\partial}{\partial u}\right|_{u=0} \Phi_u = \delta \Phi$ satisfying $\|\delta \Phi\|_{L^2}=1$ is given by
\begin{equation}\label{eqn:geodesic_solution}
\Phi_u = \left(\Phi,\delta \Phi\right) \exp u\left(\begin{array}{cc}
\Phi^\ast \delta \Phi & -\delta \Phi^\ast \delta \Phi \\
\mathrm{Id}_{2 \times 2} & \Phi^\ast \delta \Phi \end{array}\right) \mathrm{Id}_{4 \times 2} \exp(-u \Phi^\ast \delta \Phi),
\end{equation}
where $(\Phi,\delta\Phi)$ is treated as a map $\C^{4 \times 4} \rightarrow \mathcal{V}^2$.
\end{prop}

We refer the reader to \cite[Section 2.2.2]{edelman} for the derivation in the finite-dimensional real case, to \cite{harms} for a description in the real infinite-dimensional case, and to the author's dissertation \cite[Section 4.5.3]{needham} for the full details of adapting the derivation to the complex infinite-dimensional case.

Proposition \ref{prop:geodesic_equation_stiefel} has an immediate but useful corollary. Let $\Phi \in \St$ and $\delta \Phi \in T_\Phi \St$ and let $\mathcal{W}=\mathrm{span}\{\Phi, \delta \Phi\} \subset \mathcal{V}^2$. Then $\mathcal{W}$ is a subspace of complex dimension at most $4$. Moreover, $\mathcal{W}$ inherits a Hermitian inner product by restricting $\left<\cdot,\cdot\right>_{L^2}$ and this defines an embedded finite-dimensional Stiefel manifold $\mathrm{St}_2(\mathcal{W}) \subset \St$.

\begin{cor}
The geodesic with initial data $\Phi$ and $\delta \Phi$ stays in $\mathrm{St}_2(\mathcal{W})$. 
\end{cor}

Our next goal is to show that the exponential map for $\Gr$ is well-defined. The following lemma is proved in the same way as Lemma 7 of \cite{harms}, which treats the real Hilbert space case.

\begin{lem}\label{lem:horizontal_paths}
For any path $\Phi_u$, $u \in [0,1]$, in $\St$, there exists a path $A_u$ in $\mathrm{U}(2)$ so that $\Phi_u  A_u$ is horizontal with respect to the $\mathrm{U}(2)$-action.
\end{lem}

Now we note that if a geodesic in $\St$ has initial data $(\Phi,\delta \Phi)$ with $\delta \Phi \in T_\Phi^{hor}\St$, then the geodesic stays horizontal. Otherwise it could be shortened by applying the projection of Lemma \ref{lem:horizontal_paths}. We conclude that a geodesic in $\St$ with horizontal initial data represents a geodesic in $\Gr$. The next proposition follows immediately.

\begin{prop}\label{prop:exponential_map}
The exponential map $\mathrm{exp}:T_{[\Phi]} \Gr \rightarrow \Gr$ is well-defined.
\end{prop}

\subsection{Planar Loop Space}\label{sec:totally_geodesic_submanifolds}

Consider the subset $\mathcal{L}\R \subset \mathcal{L}\C$. Let $\mathrm{Gr}_2(\mathcal{L}\R)$ denote the Grassmann manifold of \emph{real} 2-planes in $\mathcal{L}\R$. The Neretin geodesics of $\mathrm{Gr}_2(\mathcal{L}\R) \subset \mathrm{Gr}_2(\mathcal{L}\C)$ (with respect to its induced $L^2$ metric) are also of the form presented in Section \ref{sec:explicit_geodesics} and it follows that $\mathrm{Gr}_2(\mathcal{L}\R)$ is a totally geodesic submanifold of $\mathrm{Gr}_2(\mathcal{L}\C)$. Moreover, the submanifold is Lagrangian---this follows by exactly the same argument as for the finite-dimensional embedding $\mathrm{Gr}_2(\R^n) \subset \mathrm{Gr}_2(\C^n)$.

The Grassmannian $\mathrm{Gr}_2(\mathcal{L}\R)$ has a curve-theoretic interpretation. Consider the \emph{moduli space of planar loops}
$$
\mathcal{N}:=\mathrm{Imm}(S^1,\R^2)/(\R^2 \times \mathrm{SO}(2) \times \R^+)=\mathrm{Imm}(S^1,\R^2)/\mathrm{Sim}(\R^2).
$$
This space was studied by Younes, Michor, Shah and Mumford in \cite{younes} for its applications to shape recognition, where it was shown that $\mathrm{Gr}_2(\mathcal{L}\R)$ with its natural $L^2$ metric is locally isometric to $\mathcal{N}$  with the elastic metric $g^{1,1}$ (see Section \ref{subsubsec:elastic_shape_analysis} for the definition).

Putting these ideas together, we have the following proposition. 

\begin{prop}\label{prop:tot_geod_lagrangian}
The moduli space of planar loops $\mathcal{N}$ embeds as a totally geodesic, Lagrangian submanifold of the moduli space of framed loops $\M$.
\end{prop}

\begin{proof}
We choose a particular embedding $\mathcal{N} \hookrightarrow \mathcal{M}$ as follows. Let $[\gamma] \in \mathcal{N}$ with $\gamma=(\gamma^1,\gamma^2)$. We map $[\gamma]$ to $[(\gamma^1,0,\gamma^2),V] \in \mathcal{M}$, where $V$ is the \emph{planar framing} constantly pointing in the $y$-direction. In complex coordinates, this is exactly the embedding  $\mathrm{Gr}_2(\mathcal{L}\R) \hookrightarrow \mathrm{Gr}_2(\mathcal{L}\C)$ by inclusion.
\end{proof}

\subsection{Sectional Curvatures}\label{sec:sectional_curvatures}

In this section we show that $\mathcal{M}$ and its quotient by $\Dif$ are nonnegatively curved with respect to $g^\op$, echoing similar results for spaces of planar curves in \cite{bauer,younes}. It was shown in Section \ref{sec:basepoint_action} that the action of $\Dif$ on $\M$ is not free, whence $\M/\Dif$ has a discrete collection of singular points. It will therefore be convenient to use the decomposition $\Dif = \mathrm{Diff}_0^+(S^1) \times S^1$, where $\mathrm{Diff}_0^+(S^1)$ is the subgroup of diffeomorphisms of $S^1=[0,2]/(0\sim 2)$ which fix $0$. We will then consider the open submanifold $\M^f \subset \M$ containing the points on which $\Dif$ acts freely and the quotient manifolds $\M/\mathrm{Diff}_0^+(S^1)$ and $\M^f/\mathrm{Diff}^+(S^1)$.

\begin{prop}
The spaces $\M/\mathrm{Diff}_0^+(S^1)$ and $\M^f/\Dif$ are manifolds.
\end{prop}

\begin{proof}
The space of arclength-parameterized framed loops $\M_1$ is a global cross-section to the free action of $\mathrm{Diff}^+_0(S^1)$ on $\M$ and it was shown in Proposition \ref{prop:arclength_parameterized_submanifolds} that $\M_1$ is a manifold. It is straightforward to show that $\M^f/\Dif$ is a manifold by identifying it with $\left(\M^f/\mathrm{Diff}_0^+(S^1)\right)/S^1$.
\end{proof}

To show that these manifolds are nonnegatively curved, we will use an immediate corollary of O'Neill's formula \cite{oneill}: if a Riemannian manifold $(M,g)$ has nonnegative sectional curvature and  $(M,g) \rightarrow (M',g')$ is a Riemannian submersion, then $(M',g')$ is also nonnegatively curved. We need to show that certain projection operators are well-defined in order to apply O'Neill's formula to this infinite-dimensional setting. In the following lemmas, the terms \emph{vertical} and \emph{horizontal} are used with respect to $\mathrm{Diff}^+_0(S^1)$-orbits. We use $\mathcal{L}_0\R$ to denote the space of real-valued loops based at $0$; this is the Lie algebra of $\mathrm{Diff}_0^+(S^1)$.

We define the \emph{curvatures of a framed curve} $(\gamma,V)$ by the formulas
$$
\kappa_1= \left<\frac{d}{ds} T, V \right> \;\; \mbox{ and } \;\; \kappa_2 = \left<\frac{d}{ds} T, T \times V \right>,
$$
where $T=\gamma'/\|\gamma'\|$. These are related to the curvature $\kappa$ of the base curve by $\kappa^2 = \kappa_1^2 + \kappa_2^2$. 

\begin{lem}\label{lem:horizontal_space}
The vertical tangent spaces of $\cl$ contain tangent vectors of the form
$$
(\delta \gamma, \delta V) = \left(\xi T, \xi (-\kappa_1 T + \mathrm{tw} W)\right),
$$
where $\xi \in \mathcal{L}_0\R$, and the horizontal space contains vectors $(\delta \gamma, \delta V)$ satisfying 
$$
\left<\frac{d^2}{ds^2} \delta \gamma, T\right> - \left<\delta V , W \right> \mathrm{tw} = 0.
$$
\end{lem}

\begin{proof}
Let $\rho_\epsilon$ be a path in $\mathrm{Diff}_0^+(S^1)$ with $\rho_0$ the identity and $\left.\frac{d}{d\epsilon}\right|_{\epsilon = 0} \rho_\epsilon = \xi \in \mathcal{L}_0 \R$. Then
$$
\left. \frac{d}{d \epsilon} \right|_{\epsilon = 0} (\gamma(\rho_\epsilon), V(\rho_{\epsilon})) = (\xi \gamma', \xi V') = (\xi \|\gamma'\| T, \xi \|\gamma'\|(-\kappa_1 T + \mathrm{tw} W)),
$$
where the second equality follows by the fact that $\frac{d}{ds} V = -\kappa_1 T + \mathrm{tw} W$. Absorbing $\|\gamma'\|$ into $\xi$ gives the characterization of the vertical space. 

A tangent vector $(\delta \gamma, \delta V)$ is horizontal if and only if
$$
g^{\open}_{(\gamma,V)}((\delta \gamma, \delta V),(\xi T, \xi(-\kappa_1 T + \mathrm{tw} W))) = 0
$$
for all $\xi \in \mathcal{L}_0\R$. Then for all $\xi$,
\begin{align}
0 &= \frac{1}{4} \int_{S^1} \left<\frac{d}{ds}\delta \gamma, \frac{d}{ds} \xi T\right> + \left<\delta V, W\right> \left<\xi (-\kappa_1 T + \mathrm{tw} W),W\right> \; \mathrm{d}s  \nonumber \\
&= \frac{1}{4} \int_{S^1} \xi \left\{-\left<\frac{d^2}{ds^2} \delta \gamma, T\right> + \left<\delta V, W\right>\mathrm{tw}\right\} \; \mathrm{d}s, \label{eqn:horizontal_space_proof}
\end{align}
where we have integrated by parts and used the orthonormality of $(T,V,W)$ to obtain the last line. The integral vanishes for every $\xi  \in \mathcal{L}_0\R$ if and only if the bracketed term in (\ref{eqn:horizontal_space_proof}) is identically zero.
\end{proof}

\begin{prop}\label{prop:tangent_space_splits}
There exist orthogonal projections from the tangent space $T_{(\gamma,V)} \cl$ to its vertical and horizontal subspaces. 
\end{prop}

The proof of the proposition requires a technical lemma, which is proved by mildly adapting the proof of \cite[Lemma 4.5]{michor}. 

\begin{lem}\label{lem:invertible_operator}
Let $(\gamma,V)$ be a framed loop. The operator $L:\mathcal{L}_0\R \rightarrow \mathcal{L} \R$ defined by
$$
L:\xi \mapsto \frac{d^2}{ds^2} \xi - (\kappa^2+\mathrm{tw}^2) \cdot \xi
$$
is invertible.
\end{lem}

\begin{proof}[Proof of Proposition \ref{prop:tangent_space_splits}]
If the projections exist then we can express an arbitrary tangent vector as
\begin{equation}\label{eqn:splitting_lemma}
(\delta \gamma, \delta V) = (\xi T, \xi(-\kappa_1 V + \mathrm{tw} W)) + (\delta \gamma^{hor} ,\delta V^{hor}),
\end{equation}
where $(\delta \gamma^{hor} ,\delta V^{hor})$ is horizontal and $\xi \in \mathcal{L}_0\R$. From (\ref{eqn:splitting_lemma}) we conclude
\begin{equation}\label{eqn:splitting_lemma_2}
\left<\frac{d^2}{ds^2} \delta \gamma, T\right> = \frac{d^2}{ds^2} \xi - \xi \kappa^2 + \left<\frac{d^2}{ds^2} \delta \gamma^{hor},T\right>.
\end{equation}
On the other hand $\left<\delta V, W\right> = \xi \mathrm{tw} + \left<\delta V^{hor}, W\right>$. Multiplying this expression by $\mathrm{tw}$ and subtracting the result from (\ref{eqn:splitting_lemma_2}) yields
$$
\left<\frac{d^2}{ds^2} \delta \gamma, T\right> - \left<\delta V, W\right>\mathrm{tw} = \frac{d^2}{ds^2} \xi - \xi \kappa^2 + \left<\frac{d^2}{ds^2} \delta \gamma^{hor},T\right> -\xi\mathrm{tw}^2 - \left<\delta V^{hor}, W\right>\mathrm{tw}.
$$
The horizontality characterization of Lemma \ref{lem:horizontal_space} simplifies this to 
$$
\left<\frac{d^2}{ds^2} \delta \gamma, T\right> - \left<\delta V, W\right>\mathrm{tw} = \frac{d^2}{ds^2}\xi - (\kappa^2 + \mathrm{tw}^2) \xi = L(\xi),
$$
where $L$ is the invertible linear operator from Lemma \ref{lem:invertible_operator}. Therefore we define
$$
\xi := L^{-1} \left(\left<\frac{d^2}{ds^2} \delta \gamma, T\right> - \left<\delta V, W\right>\mathrm{tw} \right).
$$
This gives us a well-defined projection onto the vertical space of $\cl$ and this suffices to prove the proposition.
\end{proof}

\begin{cor}\label{cor:projection_diff_M}
There exist well-defined orthogonal projections onto the vertical and horizontal tangent spaces of $\M$. 
\end{cor}

\begin{proof}
Since $\M = \cl /(\mathrm{Sim}_0 \times \bbS^1)$ is the image of a submersion with finite-dimensional fibers, we can identify $T_{[\gamma,V]}\M$ with a finite codimension subspace of $T_{(\gamma,V)}\cl$. Thus we can first project onto the vertical or horizontal space of $\cl$ using Proposition \ref{prop:tangent_space_splits}, then project onto the finite codimension subspace.
\end{proof}

Finally, we will need the following lemma. 

\begin{lem}\label{lem:geodesic_n_gons}
Let $[\Phi_1],\ldots,[\Phi_k]$ be distinct elements of $\Gr$. There exists an embedded finite-dimensional totally geodesic Grassmannian which contains every $[\Phi_j]$.
\end{lem}

\begin{proof}
Let $\mathcal{W}$ be a finite-dimensional linear subspace of $\mathcal{V}$ which contains all planes $[\Phi_j]$. By the explicit formula for the geodesics of $\Gr$ in Section \ref{sec:explicit_geodesics}, we see that the embedded finite-dimensional Grassmannian $\mathrm{Gr}_2(\mathcal{W}) \subset \Gr$ is totally geodesic.
\end{proof}

\begin{thm}\label{thm:sectional_curvature}
The spaces $\M$, $\M/\mathrm{Diff}^+_0(S^1)$ and $\M^f/\Dif$ have nonnegative sectional curvatures with respect to the metrics induced by $g^\op$.
\end{thm}

\begin{proof}
By Theorem \ref{thm:grassmannian}, $\M$ is locally isometric to $\mathrm{Gr}_2(\mathcal{V})$, and it suffices to bound the curvature of the Grassmannian to prove that $\M$ is nonnegatively curved. This can be done using classical methods (cf.\ \cite{younes}), but we will give a proof which emphasizes the role of the explicit geodesics of $\M$. Assuming that $\M$ has non-negative sectional curvature, it follows by O'Neill's formula that $\M/\mathrm{Diff}^+_0(S^1)$ does as well. Indeed, the metric $g^\op$ is reparameterization-invariant by construction, so the quotient map $\M \rightarrow \M/\mathrm{Diff}^+_0(S^1)$ is a Riemannian submersion and Corollary \ref{cor:projection_diff_M} shows that O'Neill's formula can be applied. O'Neill's formula also implies that $\M^f/\Dif$ is nonegatively curved as it is identified with the quotient $\left(\M^f/\mathrm{Diff}_0^+(S^1)\right)/S^1$, which is the image of a submersion since we are restricting to the free part $\M^f$.

Let $[\Phi] \in \Gr$  and let $\delta \Phi_1$, $\delta \Phi_2 \in T_{[\Phi]} \Gr$ be linearly independent and assume without loss of generality that each vector is $L^2$-normalized. Using Proposition \ref{prop:exponential_map}, we can choose a point $[\Phi_j]$ along the geodesic through $[\Phi]$ with velocity vector $\delta \Phi_j$ for $j=1,2$. These points can be chosen to be distinct, by the explicit description of the exponential map given in Proposition \ref{prop:geodesic_equation_stiefel}.  Next we use Corollary \ref{lem:geodesic_n_gons} to choose a totally geodesic isometrically embedded finite-dimensional Grassmannian containing $[\Phi]$, $[\Phi_1]$ and $[\Phi_2]$. In particular, this Grassmannian contains $[\Phi]$ and contains each $\delta\Phi_j$ in its tangent space. Thus the sectional curvatures of $\Gr$ and the finite dimensional Grassmannian agree at that point and plane. Finite-dimensional Grassmannians are well known to be nonnegatively curved as they are symmetric spaces of compact type, so this proves that the curvature of $\M$ is nonnegative.

\end{proof}

\section*{Acknowledgment}

Most of the work in this paper was part of my Phd.\ dissertation. I am extremely grateful to my advisor Jason Cantarella for his guidance in developing these ideas. This work would not have been possible without his help.

\end{document}